\documentclass[12pt]{amsart}
\usepackage[all]{xy}
\usepackage{amssymb,amsmath,amsfonts}
\usepackage{enumerate}
\newtheorem{theorem}{Theorem}
\newtheorem{lemma}[theorem]{Lemma}
\newtheorem{proposition}[theorem]{Proposition}
\newtheorem{corollary}[theorem]{Corollary}

\theoremstyle{definition}

\theoremstyle{remark}
\newtheorem{remark}[theorem]{Remark}

\numberwithin{equation}{section}
\numberwithin{theorem}{section}

\newcommand\leref{Lemma \ref}
\newcommand\prref{Proposition \ref}
\newcommand\coref{Corollary \ref}

\def\CC{\mathbb{C}}
\def\QQ{\mathbb{Q}}

\def\ZZ{\mathbb{Z}}
\DeclareMathOperator\End{End}

\DeclareMathOperator\Hom{Hom}

\DeclareMathOperator\Ind{Ind}

\def\vac{{\boldsymbol{1}}}  %{|0\rangle} % vacuum vector

\def\ii{\mathrm{i}} %{\sqrt{-1}}
\def\al{\alpha}
\def\be{\beta}
\def\ga{\gamma}

\def\ep{\varepsilon}

\def\la{\lambda}
\def\si{\sigma}

\def\Om{\Omega}

\def\d{\partial}
\def\<{\left\langle}
\def\>{\right\rangle}

\def\lieh{{\mathfrak{h}}}

\def\F{\mathcal{F}}
\def\W{\mathcal{W}}

\begin{document}
\title[Orbifolds of lattice vertex algebras]{Orbifolds of lattice vertex algebras under an isometry of order two}
\author{Bojko Bakalov}
\address{Department of Mathematics\\
North Carolina State University\\
Raleigh, NC 27695, USA}
\email{bojko\_bakalov@ncsu.edu}

\author{Jason Elsinger}
\address{Department of Mathematics\\
Spring Hill College\\
Mobile, AL 36695, USA}
\email{jelsinger@shc.edu}

\thanks{The first author was supported in part by a Simons Foundation grant}
%\thanks{The second author is supported in part by a JSPS Grant-In-Aid} %DMS-0707150

\date{February 17, 2015; Revised July 17, 2015}


\begin{abstract}
Every isometry $\sigma$ of a positive-definite even lattice $Q$ can be lifted to an automorphism of 
the lattice vertex algebra $V_Q$. An important problem in vertex algebra theory and conformal field theory is to classify the
representations of the $\sigma$-invariant subalgebra $V_Q^\sigma$ of $V_Q$,
known as an orbifold. In the case when $\sigma$ is an isometry of $Q$ of order two, we classify the irreducible modules of the
orbifold vertex algebra $V_Q^\sigma$ and identify them as submodules of twisted or untwisted $V_Q$-modules. 
The examples where $Q$ is a root lattice and $\sigma$ is a Dynkin diagram automorphism are presented in detail.
\end{abstract}

\maketitle

%\tableofcontents

%%%%%%%%%%%%%%%%%%%%%%%%%%%%%%%
\section{Introduction}\label{sintro}
%%%%%%%%%%%%%%%%%%%%%%%%%%%%%%%
The notion of a vertex algebra introduced by Borcherds \cite{B} provides a rigorous algebraic description of two-dimensional 
chiral conformal field theory (see e.g.\ \cite{BPZ, Go, DMS}), and is a powerful tool for studying representations of infinite-dimensional Lie algebras.
The theory of vertex algebras has been developed in \cite{FLM, K2, FB, LL, KRR} among other works. 
One of the spectacular early applications of vertex algebras was Borcherds' proof of the moonshine conjectures about the Monster group (see \cite{B2,Ga}),
which used essentially the Frenkel--Lepowsky--Meurman construction of a vertex algebra with a natural action of the Monster on it \cite{FLM}.

The FLM vertex algebra was constructed in three steps (see \cite{FLM,Ga}). First, one constructs a vertex algebra $V_Q$ from any even lattice $Q$ (see \cite{B,FLM})
by generalizing the Frenkel--Kac realization of affine Kac--Moody algebras in terms of vertex operators \cite{FK}. 
Second, given an isometry $\si$ of $Q$, one lifts it to an automorphism of $V_Q$ and constructs the so-called $\si$-twisted $V_Q$-modules \cite{D2,FFR}
that axiomatize the properties of twisted vertex operators \cite{KP,Le,FLM1}. Then the FLM vertex algebra is defined as the direct sum
of the $\si$-invariant part $V_Q^\si$ and one of the $\si$-twisted $V_Q$-modules, in the special case when $\si=-1$ and $Q$ is the Leech lattice.

In general, if $\si$ is an automorphism of a vertex algebra $V$, the $\si$-invariants $V^\si$ form a subalgebra of $V$ known as an orbifold, which
is important in conformal field theory (see e.g.\ \cite{DVVV, KT, DLM2} among many other works). Every $\si$-twisted representation of $V$ becomes untwisted when
restricted to $V^\si$. It is a long-standing conjecture that all irreducible $V^\si$-modules are obtained by restriction from twisted or untwisted $V$-modules. 
Under certain assumptions, this conjecture has been proved recently by M.~Miyamoto \cite{M1,M2,M3}, and he
has also shown that the vertex algebra $V^\si$ is rational, i.e., every (admissible) module is a direct sum of irreducible ones.

In this paper, we are concerned with the case when $Q$ is a positive-definite even lattice and $\si$ is an isometry of $Q$ of order two. We classify and construct
explicitly all irreducible modules of the orbifold vertex algebra $V_Q^\si$, and we realize them as submodules of twisted or untwisted $V_Q$-modules. In the important
special case when $\si=-1$, the classification was done previously by Dong--Nagatomo and Abe--Dong \cite{DN,AD}.
Our approach is to restrict from $V_Q^\si$ to $V_L^\si$ where $L$ is the sublattice of $Q$ spanned by eigenvectors of $\si$. The subalgebra $V_L^\si$
factors as a tensor product $V_{L_+}\otimes V_{L_-}^\si$, where $L_\pm$ consists of $\al\in Q$ with $\si\al=\pm\al$, respectively.
By the results of \cite{DLM3,ABD,DJL}, every $V_L^\si$-module is a direct sum of irreducible ones. We describe explicitly the irreducible $V_L^\si$-modules using
\cite{FHL,DN,AD}, and then utilize the intertwining operators among them \cite{A1,ADL} to determine the irreducible $V_Q^\si$-modules.
A new ingredient is the introduction of a sublattice $\bar{Q}$ of $Q$ such that $V_Q^\si = V_{\bar Q}^\si$ and $\si$
has order $2$ as an automorphism of $V_{\bar Q}$ (while in general $\si$ has order $4$ on $V_Q$).

Our work was motivated by the examples when $Q$ is the root lattice of a simply-laced simple Lie algebra and $\si$ is a Dynkin diagram automorphism,
which are related to twisted affine Kac--Moody algebras \cite{K1}. We also consider the case when $\si$ is the negative of a Dynkin diagram automorphism,
which is relevant to the Slodowy generalized intersection matrix Lie algebras \cite{S,Be}. 

The paper is organized as follows. In Section \ref{sva}, we briefly review lattice vertex algebras and their twisted modules, and the results for $\si=-1$ that we need.
Our main results are presented in Section \ref{sirr}. The examples when $Q$ is a root lattice of type ADE are discussed in detail in Section \ref{sroot}.

%%%%%%%%%%%%%%%%%%%%%%%%%%%%%%%
\section{Vertex algebras and their twisted modules}\label{sva}
%%%%%%%%%%%%%%%%%%%%%%%%%%%%%%%

In this section, we briefly review lattice vertex algebras and their twisted modules.
We also recall the case when $\si=-1$, which will be used essentially for the general case.
Good general references on vertex algebras are \cite{FLM, K2, FB, LL, KRR}. 

%%%%%%%%%%%%%%%%%%%%%%%%%%%%%
\subsection{Vertex algebras}\label{vert}
%%%%%%%%%%%%%%%%%%%%%%%%%%%%%

Recall that a \emph{vertex algebra} 
is a vector space $V$ with a distinguished vector $\vac\in V$ 
(vacuum vector), together with a linear map 
(state-field correspondence)
\begin{equation}\label{vert2}
Y(\cdot,z)\cdot \colon V \otimes V \to V(\!(z)\!) = V[[z]][z^{-1}] \,.
\end{equation}
Thus, for every $a\in V$, we have the \emph{field}
$Y(a,z) \colon V \to V(\!(z)\!)$. This field can be viewed as
a formal power series from $(\End V)[[z,z^{-1}]]$, which 
involves only finitely many negative powers of $z$ when
applied to any vector.

The coefficients in front of powers of $z$ in this expansion are known as the
\emph{modes} of $a$:
\begin{equation}\label{vert4}
Y(a,z) = \sum_{n\in\ZZ} a_{(n)} \, z^{-n-1} \,, \qquad
a_{(n)} \in \End V \,.
\end{equation}
%The \emph{formal residue} $\Res_z$ of a formal power series is defined as the coefficient of $z^{-1}$, and
%\begin{equation}\label{fres}
%a_{(n)} = \Res_z z^n Y(a,z) \,.
%\end{equation}
The vacuum vector $\vac$ plays the role of an identity in the sense that
\begin{equation*}%\label{vert6}
a_{(-1)}\vac = \vac_{(-1)} a = a \,, \qquad 
a_{(n)}\vac = 0 \,, \quad n\geq 0 \,.
\end{equation*}
In particular, $Y(a,z)\vac \in V[[z]]$ is regular at $z=0$, and
its value at $z=0$ is equal to $a$.

The main axiom for a vertex algebra is the \emph{Borcherds identity}
(also called Jacobi identity \cite{FLM})
satisfied by the modes:
\begin{equation}\label{vert5}
\begin{split}
\sum_{j=0}^\infty & \binom{m}{j} (a_{(n+j)}b)_{(k+m-j)}c
= \sum_{j=0}^\infty \binom{n}{j} (-1)^j 
a_{(m+n-j)}(b_{(k+j)}c)
\\
&- \sum_{j=0}^\infty \binom{n}{j} (-1)^{j+n} \, b_{(k+n-j)}(a_{(m+j)}c) \,,
\end{split}
\end{equation}
where $a,b,c \in V$. %and $p(a) \in\ZZ/2\ZZ$ denotes the parity of $a$.
Note that the above sums are finite, because
$a_{(n)}b = 0$ for sufficiently large~$n$.

We say that a vertex algebra $V$ is (strongly) \emph{generated} by a subset $S\subset V$ if $V$
is linearly spanned by the vacuum $\vac$ and all elements of the form
${a_1}_{(n_1)} \cdots {a_k}_{(n_k)} \vac$, where $k\geq1$, $a_i \in S$, $n_i<0$.

%%%%%%%%%%%%%%%%%%%%%%%%%%%%%%%%%%
\subsection{Twisted representations of vertex algebras}\label{twrep}
%%%%%%%%%%%%%%%%%%%%%%%%%%%%%%%%%%

A \emph{representation} of a vertex algebra $V$, or a $V$-\emph{module}, is a vector space $M$ endowed with a
linear map $Y(\cdot,z)\cdot \colon V \otimes M \to M(\!(z)\!)$
(cf.\ \eqref{vert2}, \eqref{vert4}) such that the Borcherds identity
\eqref{vert5} holds for $a,b\in V$, $c\in M$ (see \cite{FB, LL, KRR}).

Now let $\si$ be an automorphism of $V$ of finite order $r$. Then $\si$ is diagonalizable.
In the definition of a \emph{$\si$-twisted representation} $M$ of $V$ \cite{FFR, D2}, the image of the
above map $Y$ is allowed to have nonintegral rational powers of $z$,
so that
\begin{equation*}%\label{twlat1}
Y(a,z) = \sum_{n\in p+\ZZ} a_{(n)} \, z^{-n-1} \,, \qquad
\text{if} \quad \si a = e^{-2\pi\ii p} a \,, \; p\in\frac1r\ZZ \,,
\end{equation*}
where $a_{(n)} \in \End M$.
%Equivalently, the monodromy around $z=0$ is given by the action of $\si$:
%\begin{equation}\label{twrep2}
%Y(\si a,z) = Y(a, e^{2\pi\ii}z) \,, \qquad a\in V \,.
%\end{equation}
The Borcherds identity \eqref{vert5} satisfied by the modes remains the
same in the twisted case, provided that $a$ is an eigenvector of $\si$.
An important consequence of the Borcherds identity is the
\emph{locality} property \cite{DL1,Li}:
\begin{equation}\label{locpr}
(z-w)^{N} [Y(a,z), Y(b,w)] = 0 
\end{equation}
for sufficiently large $N$ depending on $a,b$ (one can take $N$ to be such that
$a_{(n)}b=0$ for $n\ge N$).

The following result provides a rigorous interpretation of the \emph{operator product expansion} in conformal field theory
(cf.\ \cite{Go, DMS}) in the case of twisted modules.

\begin{proposition}[\cite{BM}]\label{pnprod}
Let\/ $V$ be a vertex algebra, $\si$ an automorphism of\/ $V$, and\/ $M$ a $\si$-twisted representation of\/ $V$.
Then
\begin{equation}\label{locpr3}
\frac1{k!} \d_{z}^k \Bigl( (z-w)^{N} \, Y(a,z) Y(b,w) c \Bigr)\Big|_{z=w}
= Y(a_{(N-1-k)} b, w) c
\end{equation}
for all\/ $a,b\in V$, $c\in M$, $k\geq0$, and sufficiently large\/ $N$.
%We can take\/ $N$ such that\/ \eqref{locpr} holds.
Conversely, \eqref{locpr} and \eqref{locpr3} imply the Borcherds identity \eqref{vert5}.
\end{proposition}

Recall from \cite{FHL} that if $V_1$ and $V_2$ are vertex algebras, their tensor product is again a vertex algebra with 
\begin{equation*}%\label{tensprod}
Y(v_1\otimes v_2,z) = Y(v_1,z) \otimes Y(v_2,z) \,, \qquad v_i \in V_i \,.
\end{equation*}
Furthermore, if $M_i$ is a $V_i$-module, then the above formula defines the structure of a $(V_1\otimes V_2)$-module
on $M_1\otimes M_2$ (see \cite{FHL}). This is also true for twisted modules. %, and we provide the proof for completeness.

\begin{lemma}\label{ltensprod1}
For\/ $i=1,2$, let\/ $V_i$ be a vertex algebra, $\si_i$ an automorphism of\/ $V_i$, and\/ $M_i$ a $\si_i$-twisted representation of\/ $V_i$.
Then\/ $M_1\otimes M_2$ is a\/ $(\si_1\otimes \si_2)$-twisted module over\/ $V_1\otimes V_2$.
\end{lemma}
\begin{proof}
By \prref{pnprod}, it is enough to check \eqref{locpr} and \eqref{locpr3} for $a=a_1\otimes a_2$ and $b=b_1\otimes b_2$, given that they hold
for the pairs $a_1,b_1\in V_1$ and $a_2,b_2\in V_2$. This is done by a straightforward calculation.
\end{proof}

%%%%%%%%%%%%%%%%%%%%%%%%%%%%
\subsection{Lattice vertex algebras}\label{lat}
%%%%%%%%%%%%%%%%%%%%%%%%%%%%

Let $Q$ be an \emph{integral lattice}, i.e., a free abelian group of finite rank equipped with a symmetric
nondegenerate bilinear form $(\cdot|\cdot) \colon Q\times Q\to\ZZ$. We will assume that $Q$ is \emph{even},
i.e., $|\al|^2=(\al|\al) \in2\ZZ$ for all $\al\in Q$.
We denote by   
$\lieh = \CC\otimes_\ZZ Q$ the corresponding complex 
vector space considered as an abelian Lie algebra, and extend the bilinear form to it.

The \emph{Heisenberg algebra}
$\hat\lieh = \lieh[t,t^{-1}] \oplus \CC K$
is the Lie algebra with brackets
\begin{equation}\label{heis1}
[a_m,b_n] = m \delta_{m,-n} (a|b) K \,, \qquad
a_m=at^m \,,
\end{equation}
where $K$ is central.
Its irreducible highest-weight representation 
\begin{equation*}
\F = \Ind^{\hat\lieh}_{\lieh[t]\oplus\CC K} \CC \cong S(\lieh[t^{-1}]t^{-1})
\end{equation*}
on which $K=1$ is known as the (bosonic) \emph{Fock space}.

Following \cite{FK, B}, we consider a
$2$-cocycle $\ep\colon Q \times Q \to \{\pm1\}$ 
such that
\begin{equation}\label{lat2}
\ep(\al,\al) = (-1)^{|\al|^2(|\al|^2+1)/2}  \,,
\qquad \al\in Q \,,
\end{equation}
and the associative algebra $\CC_\ep[Q]$ with basis
$\{ e^\al \}_{\al\in Q}$ and multiplication
\begin{equation}\label{lat1}
%\CC_\ep[Q]=\Span_\CC \{ e^\al | \al\in Q \} \,, \qquad
e^\al e^\be = \ep(\al,\be) e^{\al+\be} \,.
\end{equation}
Such a $2$-cocycle $\ep$ is unique up to equivalence and can be chosen
to be bimultiplicative. In addition,
\begin{equation}\label{lat22}
\ep(\al,\be) \ep(\be,\al) = (-1)^{(\al|\be)}  \,, 
\qquad \al,\be\in Q \,.
\end{equation}

The \emph{lattice vertex algebra} %\cite{B, FLM, K2, FB, LL}
associated to $Q$ is defined as $V_Q=\F\otimes\CC_\ep[Q]$,
where the vacuum vector is $\vac=1\otimes e^0$.
We let the Heisenberg algebra act on $V_Q$ so that
\begin{equation*}%\label{lat3}
a_n e^\be = \delta_{n,0} (a|\be) e^\be \,, \quad n\geq0 \,, \qquad
a\in\lieh \,.
\end{equation*}
%In other words, as a module over $\hat\lieh$, it is a direct sum of highest-weight representations with highest weights $\al\in Q$.
%
The state-field correspondence on $V_Q$ is uniquely determined by the
generating fields:
\begin{align}\label{lat4}
Y(a_{-1}\vac,z) &= \sum_{n\in\ZZ} a_n \, z^{-n-1} \,, \qquad a\in\lieh \,,
\\ \label{lat5}
Y(e^\al,z) &= e^\al z^{\al_0} 
\exp\Bigl( \sum_{n<0} \al_n \frac{z^{-n}}{-n} \Bigr) 
\exp\Bigl( \sum_{n>0} \al_n \frac{z^{-n}}{-n} \Bigr) \,,
\end{align}
where $z^{\al_0} e^\be = z^{(\al|\be)} e^\be$.

Notice that $\F\subset V_Q$ is a vertex subalgebra, which we call the \emph{Heisenberg vertex algebra}.
The map $\lieh\to\F$ given by $a\mapsto a_{-1}\vac$ is injective.
{}From now on, we will slightly abuse the notation and identify
$a\in\lieh$ with $a_{-1}\vac \in\F$; then $a_{(n)}=a_n$ for all $n\in\ZZ$. 

%When the lattice $Q$ is non-integral, the above formulas still make sense but
%involve non-integral powers of $z$. In this case, $V_Q$ is a 
%\emph{generalized} vertex algebra \cite{FFR, DL1, BK2}. 

%%%%%%%%%%%%%%%%%%%%%%%%%%%%%%%%%
\subsection{Twisted Heisenberg algebra}\label{twheis}
%%%%%%%%%%%%%%%%%%%%%%%%%%%%%%%%%

Every automorphism $\si$ of $\lieh$ preserving the bilinear form induces automorphisms
of $\hat\lieh$ and $\F$, which will be denoted again as $\si$. As before, assume that $\si$
has a finite order $r$.
The action of $\si$ can be extended to $\lieh[t^{1/r},t^{-1/r}] \oplus\CC K$ by letting
\begin{equation*}%\label{twheis1}
\si(a t^m) = \si(a) e^{2\pi\ii m}t^m \,, \quad \si(K)=K \,,
\qquad a\in\lieh \,, \; m\in\frac1r\ZZ \,.
\end{equation*}
The \emph{$\si$-twisted Heisenberg algebra} 
$\hat\lieh_\si$ is defined as the set of all $\si$-invariant elements
(see e.g.\ \cite{KP,Le,FLM1}).
In other words, $\hat\lieh_\si$ is spanned over $\CC$ by $K$ and
the elements $a_m = at^m$ such that $\si a = e^{-2\pi\ii m} a$.
This is a Lie algebra with bracket 
(cf.\ \eqref{heis1})
\begin{equation*}%\label{twheis5}
[a_m,b_n] = m \delta_{m,-n} (a|b) K \,, \qquad a,b\in\lieh \,, \;\; m,n\in \frac1r\ZZ \,.
\end{equation*}
Let $\hat\lieh_\si^\ge$ (respectively, $\hat\lieh_\si^<$) be the abelian subalgebra of 
$\hat\lieh_\si$ spanned by all elements $a_m$ with $m\geq0$ 
(respectively, $m<0$). 
%Elements of $\hat\lieh_\si^\ge$ are called \emph{annihilation operators}, 
%while elements of $\hat\lieh_\si^<$ \emph{creation operators}.

The \emph{$\si$-twisted Fock space} is defined as
\begin{equation}\label{twheis2}
\F_\si = \Ind^{\hat\lieh_\si}_{\hat\lieh_\si^\ge \oplus\CC K} \CC \cong S(\hat\lieh_\si^<) \,,
\end{equation}
where $\hat\lieh_\si^\ge$ acts on $\CC$ trivially and $K$ acts as the identity operator.
Then $\F_\si$ is an irreducible highest-weight representation of $\hat\lieh_\si$, and
has the structure of a $\si$-twisted representation of the vertex algebra $\F$
(see e.g.\ \cite{FLM,KRR}). This structure can be described as follows. 
We let $Y(\vac,z)$ be the identity operator and
\begin{equation*}%\label{twheis3}
Y(a,z) = \sum_{n\in p+\ZZ} a_{n} \, z^{-n-1} \,, \qquad
a\in\lieh \,, \;\; \si a = e^{-2\pi\ii p} a \,, \;\;  p\in\frac1r\ZZ \,,
\end{equation*}
%where $p\in\frac1r\ZZ$ (cf.\ \eqref{twlat1}),
and we extend $Y$ to all $a\in\lieh$ by linearity.
%These satisfy the locality property \eqref{locpr} because
%\begin{equation}\label{twheis4}
%[Y(a,z_1), Y(b,z_2)] = (a|b) \, \d_{z_2} \bigl( z_1^{-p} z_2^{p} \de(z_1,z_2) \bigr) \,.
%\end{equation}
The action of $Y$ on other elements of $\F$ is then determined by applying several times the product
formula \eqref{locpr3}.
More explicitly, $\F$ is spanned by elements of the form $a^1_{m_1}\cdots a^k_{m_k} \vac$
where $a^j\in\lieh$, and we have:
\begin{align*}
Y& ( a^1_{m_1}\cdots a^k_{m_k} \vac, z ) c
\\
&= \prod_{j=1}^k \d_{z_j}^{(N-1-m_j)} 
\Bigl( \prod_{j=1}^k (z_{j}-z)^N \; Y(a^1,z_1) \cdots Y(a^k,z_k) c \Bigr)\Big|_{z_1=\cdots=z_k=z}
\end{align*}
for all $c\in\F_\si$ and sufficiently large $N$.
In the above formula, we use the divided-power notation $\d^{(n)} = \d^n / n!$.

%%%%%%%%%%%%%%%%%%%%%%%%%%%%%%%%%%%%%
\subsection{Twisted representations of lattice vertex algebras}\label{twlat}
%%%%%%%%%%%%%%%%%%%%%%%%%%%%%%%%%%%%%

%Now let $Q$ be a positive-definite \emph{even} lattice, which means that
%$|\al|^2$ is a positive even integer for every nonzero $\al\in Q$.
Let $\si$ be an automorphism (or \emph{isometry})
of the lattice $Q$ of finite order $r$, so that
\begin{equation}\label{twlat1a}
(\si\al | \si\be) = (\al | \be) \,, \qquad \al,\be\in Q \,.
\end{equation}
The uniqueness of the cocycle $\ep$ and \eqref{twlat1a}, \eqref{lat22} imply that
\begin{equation}\label{twlat3}
\eta(\al+\be) \ep(\si\al,\si\be) = \eta(\al)\eta(\be) \ep(\al,\be)
\end{equation}
for some function $\eta\colon Q\to\{\pm1\}$. 
\begin{lemma}\label{leta}
Let\/ $L$ be a sublattice of\/ $Q$ such that\/ $\ep(\si\al,\si\be) = \ep(\al,\be)$
for\/ $\al,\be\in L$. Then there exists a function $\eta\colon Q\to\{\pm1\}$
satisfying \eqref{twlat3} and\/ $\eta(\al)=1$ for all\/ $\al\in L$.
\end{lemma}
\begin{proof}
First observe that, by \eqref{lat2} and \eqref{twlat1a}, \eqref{twlat3} for $\al=\be$, we have
$\eta(2\al)=1$ for all $\al\in Q$. Since, by bimultiplicativity, $\ep(2\al,\be)=1$, we obtain that
$\eta(2\al+\be)=\eta(\be)$ for all $\al,\be$. Therefore,
$\eta$ is defined on $Q/2Q$. If $\al_1,\dots,\al_\ell$ is any $\ZZ$-basis for $Q$, we can set
all $\eta(\al_i)=1$ and then $\eta$ is uniquely extended to the whole $Q$ by \eqref{twlat3}.
We can pick a $\ZZ$-basis for $Q$ so that $d_1\al_1,\dots,d_m\al_m$ is a $\ZZ$-basis for $L$,
where $m\le\ell$ and $d_i\in\ZZ$. Then the extension of $\eta$ to $Q$ will satisfy
$\eta(\al)=1$ for all $\al\in L$.
\end{proof}
In particular, $\eta$ can be chosen such that
\begin{equation}\label{twlat2}
\eta(\al)=1 \,,\qquad \al\in Q\cap\lieh_0 \,,
\end{equation}
where $\lieh_0$ denotes the subspace of $\lieh$ consisting of vectors fixed under $\si$.
Then $\si$ can be lifted to an automorphism of 
the lattice vertex algebra $V_Q$ by setting
\begin{equation}\label{twlat4}
\si(a_n)=\si(a)_n \,, \quad \si(e^\al)=\eta(\al) e^{\si\al} \,,
\qquad a\in\lieh \,, \; \al\in Q \,.
\end{equation}
\begin{remark}\label{rqbar}
The order of $\si$ is $r$ or $2r$ when viewed as an automorphism of $V_Q$, where $r$ is the order of $\si$ on $Q$.
The set $\bar Q$ of $\al\in Q$ such that $\si^r(e^\al)=e^\al$ is a sublattice of $Q$ of index $1$ or $2$, and we have
$(V_Q)^{\si^r} = V_{\bar Q}$.
\end{remark}

We will now recall the construction of irreducible $\si$-twisted $V_Q$-modules (see \cite{KP,Le,D2,BK}).
Introduce the group $G = \CC^\times \times \exp\lieh_0 \times Q$ consisting of elements $c \, e^h U_\al$ 
($c\in\CC^\times$, $h\in\lieh_0$, $\al\in Q$) with multiplication
\begin{align*}
%\label{tig1}
e^h e^{h'} &= e^{h+h'} \,,
\\
%\label{tig2}
e^h U_\al e^{-h} &= e^{(h|\al)} U_\al \,,
\\
%\label{tig3}
U_\al U_\be &= \ep(\al,\be) B_{\al,\be}^{-1} \, U_{\al+\be} \,,
\end{align*}
where
\begin{equation*}%\label{twlat8}
B_{\al,\be} = 
r^{ -(\al|\be) } \prod_{k=1}^{r-1} \bigl(1 - e^{2\pi\ii k/r} \bigr)^{ (\si^k\al|\be) } \,.
\end{equation*}
Let
\begin{equation*}%\label{twlat9}
C_\al = \eta(\al) U_{\si\al}^{-1} U_\al e^{ 2\pi\ii(b_\al+\pi_0\al) } \,, \qquad
b_\al=\frac12 \bigl( |\pi_0\al|^2-|\al|^2 \bigr) \,,
\end{equation*}
where 
$\pi_0$ is the projection of $\lieh$ onto $\lieh_0$.
Then $C_\al C_\be = C_{\al+\be}$ and all $C_\al$ are in the center of $G$.
We define $G_\si$ to be the quotient group $G/\{C_\al\}_{\al\in Q}$.

Consider an irreducible representation $\Om$ of $G_\si$.
Such representations are parameterized by the set $(Q^*/Q)^\si$ of $\si$-invariants in $Q^*/Q$,
i.e., by $\la+Q$ such that $\la\in Q^*$ and $(1-\si)\la\in Q$ (see \cite[Proposition 4.4]{BK}).
Furthermore, the action of $\exp\lieh_0$ on $\Om$ is semisimple:
\begin{equation*}%\label{twlat10}
\Om=\bigoplus_{\mu\in\pi_0(Q^*)} \Om_\mu \,,
\end{equation*}
where 
\begin{equation*}%\label{ommu}
\Om_\mu = \{ v\in\Om \,|\, e^h v = e^{(h|\mu)} v 
\;\;\text{for}\;\; h\in\lieh_0 \} \,. 
\end{equation*}

Then $\F_\si\otimes\Om$ is an irreducible $\si$-twisted $V_Q$-module with an action defined as follows.
We define $Y(a,z)$ for $a\in\lieh$ as before, %(see \eqref{twheis3}), 
and for $\al\in Q$ we let
\begin{equation}\label{twlat12}
Y(e^\al,z) = 
\exp\Biggl( \sum_{  n\in\frac1r\ZZ_{<0} } \al_n \frac{z^{-n}}{-n} \Biggr)
\exp\Biggl( \sum_{  n\in\frac1r\ZZ_{>0} } \al_n \frac{z^{-n}}{-n} \Biggr)
\otimes U_\al z^{b_\al+\pi_0\al} \,.
\end{equation}
Here the action of $z^{\pi_0\al}$ is given by $z^{\pi_0\al} v = z^{(\pi_0\al|\mu)} v$ for $v\in\Om_\mu$,
and $(\pi_0\al|\mu) \in \frac1r\ZZ$.
The action of $Y$ on all of $V_Q$ can be obtained by applying the product
formula \eqref{locpr3}.

By \cite[Theorem 4.2]{BK}, every irreducible $\si$-twisted $V_Q$-module is obtained in this way,
and every $\si$-twisted $V_Q$-module is a direct sum of irreducible ones.
In the special case when $\si=1$, we get Dong's Theorem that the irreducible $V_Q$-modules
are classified by $Q^*/Q$ (see \cite{D1}). Explicitly, they are given by:
\begin{equation*}%\label{twlat14}
V_{\la+Q} = \F\otimes\CC_\ep[Q] e^\la\,, \qquad \la\in Q^* \,.
\end{equation*}

When the lattice $Q$ is written as an orthogonal direct sum of sublattices, $Q=L_1\oplus L_2$, we have a natural isomorphism
$V_Q \cong V_{L_1} \otimes V_{L_2}$. The following lemma shows that if $L_1$ and $L_2$ are $\si$-invariant, there is
a correspondence of irreducible twisted modules (cf.\ \cite{FHL} and \leref{ltensprod1}).

\begin{lemma}\label{ltensprod2}
Let\/ $Q$ be an even lattice, $Q=L_1\oplus L_2$ an orthogonal direct sum, and\/ $\si$ an automorphism of\/ $Q$ such that\/ $\si(L_i)\subseteq L_i \; (i=1,2)$.
Set\/ $\si_i=\si|_{L_i}$. Then every irreducible\/ $\si$-twisted\/ $V_Q$-module\/ $M$ is a tensor product, $M\cong M_1\otimes M_2$, where\/
$M_i$ is an irreducible\/ $\si_i$-twisted\/ $V_{L_i}$-module.
\end{lemma}
\begin{proof}
This follows from the classification of irreducible $\si$-twisted $V_Q$-modules described above. Indeed, $Q=L_1\oplus L_2$ gives rise to a decomposition
of the Heisenberg Lie algebra as a direct sum $\hat\lieh=\hat\lieh_1\oplus\hat\lieh_2$. We also have a similar decomposition for the corresponding
twisted Heisenberg algebras. Then for the twisted Fock spaces, we get
$\F_\si=\F_{\si_1}\otimes \F_{\si_2}$. Similarly, the group $G$ is a direct product of its subgroups $G_1$ and $G_2$ associated to the lattices $L_1$ and $L_2$,
respectively.
\end{proof}

%%%%%%%%%%%%%%%%%%%%%%%%%%%%%%%%%%%%%
\subsection{The case $\si=-1$}\label{tcasesi-1}
%%%%%%%%%%%%%%%%%%%%%%%%%%%%%%%%%%%%%

Now we will review what is known in the case when $\si=-1$, which will be used essentially in our treatment of the
general case. In this subsection, we will denote the even integral lattice by $L$ instead of $Q$. As before, let 
$\lieh=\CC\otimes_\ZZ L$ be the corresponding complex vector space.

Observe that $\lieh_0=0$, $\pi_0=0$, and we can assume that $\eta(\al)=1$
for all $\al\in L$. Hence, the automorphism $\si$ acts on $V_L$ by
\begin{equation}\label{siact1}
\si\left(h^1_{(-n_1)}\cdots h^k_{(-n_k)}e^\al\right)=(-1)^kh^1_{(-n_1)}\cdots h^k_{(-n_k)}e^{-\al}
\end{equation}
for $h^i\in\lieh$, $n_i\in\ZZ_{\geq 0}$ and $\al\in L$.
The group $G$ consists of $cU_\al$ $(c\in\CC^\times,\al\in L)$, and its center consists of $cU_\al$ with $\al\in 2L^*\cap L$.
Then $G_\si$ is the quotient of $G$ by $\{U_\al^{-1}U_{-\al}\}_{\al\in L}$.
The twisted Heisenberg algebra is $\hat\lieh_\si = \lieh[t,t^{-1}] t^{1/2} \oplus\CC K$.

For any $G_\si$-module $T$, define $V_L^T=\F_\si\otimes T$, where $\F_\si$ is the twisted Fock space (cf.\ \eqref{twheis2}).
By \cite{D2}, every irreducible $\si$-twisted $V_L$-module is isomorphic to $V_L^T$ for some irreducible $G_\si$-module $T$
on which $cU_0$ acts as $c I$, where $I$ is the identity operator. Such modules $T$ can be described equivalently as $G$-modules, 
on which $cU_0=cI$ and $U_\al=U_{-\al}$. The irreducible ones are determined by the central characters $\chi$ such that
$\chi(U_\al)=\chi(U_{-\al})$ for $\al\in 2L^*\cap L$.
We have:
\[
U_\al^2 = U_\al U_{-\al} = \ep(\al,-\al) B_{\al,-\al}^{-1} U_0 \,,
\]
which implies that
\begin{equation}\label{chiual}
\chi(U_\al) = s(\al) e^{\pi\ii |\al|^2 (|\al|^2+1)/2} 2^{-|\al|^2} \,,
\end{equation}
where $s(\al) \in\{\pm1\}$ satisfies %$s(\al)=s(-\al)$ and 
$s(\al+\be)=s(\al)s(\be)$. All such maps $s$
have the form 
\[
s(\al)=(-1)^{(\al|\mu)} 
\]
for some $\mu\in (2L^*\cap L)^*$. The corresponding central characters $\chi$ will be denoted as $\chi_\mu$,
and the corresponding $G_\si$-module $T$ as $T_\mu$.

We define an action of $\si$ on $V_L^T$ by
\begin{equation}\label{siact2}
\si\left(h^1_{(-n_1)}\cdots h^k_{(-n_k)}t\right)=(-1)^kh^1_{(-n_1)}\cdots h^k_{(-n_k)}t
\end{equation}
for $h^i\in\lieh$, $n_i\in\frac{1}{2}+\ZZ_{\geq 0}$ and $t\in T$. The eigenspaces for $\si$ are denoted $V_L^{T,\pm}$. 
Then we have
\begin{equation}\label{siact3}
\si\bigl( Y(a,z) v \bigr) = Y(\si a,z) (\si v) \,, \qquad a\in V_L \,, \; v\in V_L^T \,,
\end{equation}
which implies that $\si$ is an automorphism of $V_L^\si$-modules. In particular, $V_L^{T,\pm}$
are $V_L^\si$-modules. Similarly, we define an action of $\si$ on the untwisted $V_L$-modules $V_{\la+L}$ by
\eqref{siact1} for $\al\in\la+L$.
Note that $\si V_{\la+L} \subseteq V_{-\la+L}$. Hence, if $\la\in L^*$, $2\la\in L$, then
the eigenspaces $V_{\la+L}^\pm$ are $V_L^\si$-modules. On the other hand, if $2\la\not\in L$, then
$V_{\la+L} \cong V_{-\la+L}$ as $V_L^\si$-modules.

\begin{theorem}[\cite{DN, AD}]\label{AD}
Let\/ $L$ be a positive-definite even lattice and\/ $\si=-1$ on\/ $L$.
Then any irreducible admissible\/ $V_L^\si$-module is isomorphic to one of the following$:$
\[
V_{\la+L}^\pm\quad (\la\in L^*, \,2\la\in L),\quad V_{\la+L}\quad (\la\in L^*, \,2\la\notin L),\quad V_L^{T,\pm},
\]
where\/ $T$ is an irreducible\/ $G_\si$-module.
\end{theorem}

Next, we will discuss intertwining operators between the irreducible $V_L^\si$-modules. 
For a vector space $U$, denote by
\begin{equation*}
U\{z\}=\Bigl\{\sum_{n\in\QQ}v_{(n)}z^{-n-1}\Big|v_{(n)}\in U\Bigr\}
\end{equation*}
the space of $U$-valued formal series involving rational powers of $z$.
Let $V$ be a vertex algebra, and $M_1, M_2, M_3$ be $V$-modules, which are not necessarily distinct. 
Recall from \cite{FHL} that an \emph{intertwining operator} of type
$\displaystyle\binom{M_3}{M_1\;M_2}$
is a linear map \,\,$\mathcal{Y}\colon M_1\otimes M_2\to M_3\{z\}$, or equivalently,
\begin{align*}
\mathcal{Y}\colon M_1&\to\Hom(M_2,M_3)\{z\} \,,\\
v&\mapsto \mathcal{Y}(v,z)=\sum_{n\in\QQ}v_{(n)}z^{-n-1} \,,\quad v_{(n)}\in\Hom(M_2,M_3)
\end{align*}
such that $v_{(n)}u=0$ for $n\gg 0$, and the Borcherds identity \eqref{vert5} holds for $a\in V, b\in M_1$ and $c\in M_2$ with $k\in\QQ$ and $m, n\in\ZZ$.
The intertwining operators of type
$\displaystyle\binom{M_3}{M_1\;M_2}$
form a vector space denoted $\mathcal{V}_{M_1\,M_2}^{M_3}$. 
The {\it fusion rule} associated with an algebra $V$ and its modules $M_1, M_2, M_3$ is 
$N_{M_1\,M_2}^{M_3}=\dim\mathcal{V}_{M_1\,M_2}^{M_3}$.

The fusion rules for $V_L^\si$ and its irreducible modules
were calculated in \cite{A1, ADL} to be either zero or one. 
In order to present their theorem, we first introduce some additional notation.
We assume that $\la\in L^*$ is such that $2\la\in L$, and we let
\begin{align}
%\label{clamu}
%c(\la,\mu)&=(-1)^{(\la|\mu)+|\la|^2|\mu|^2} \,, \\
\label{pi}
\pi_{\la,\mu} %&=(-1)^{(\la|\mu)}c(\la,\mu)
&=(-1)^{|\la|^2|\mu|^2} \,, \qquad \la,\mu\in L^* \,, \\
\label{c}
c_\chi(\la)&=(-1)^{(\la|2\la)}\ep(\la,2\la)s(2\la) \,.
\end{align}
%Note that $c_\chi(\la)=\chi(e_{2\la})$ when the lattice $\ZZ\la+L$ is integral. 
For any central character $\chi$ of $G_\si$, let $\chi^{(\la)}$ be the central character defined by 
\begin{equation}\label{newchi}
\chi^{(\la)}(U_\al)=(-1)^{(\al|\la)}\chi(U_\al) \,,
\end{equation}
and set $T^{(\la)}_\chi=T_{\chi^{(\la)}}$.
Note that when $\chi=\chi_\mu$ and $T=T_\mu$, we have $\chi_\mu^{(\la)} = \chi_{\la+\mu}$ and
$T_\mu^{(\la)} = T_{\la+\mu}$.

The following theorem is a special case of Theorem 5.1 from \cite{ADL}.

\begin{theorem}[\cite{ADL}]\label{ADL}
Let\/ $L$ be a positive-definite even lattice and\/ $\la\in L^*\cap\frac{1}{2}L$. 
Then for two irreducible\/ $V_L^\si$-modules\/ $M_2,M_3$ and for\/ $\epsilon\in\{\pm\}$,
the fusion rule of type\/ $\displaystyle\binom{M_3}{V_{\la+L}^{\epsilon}\;M_2}$
is equal to\/ $1$ if and only if the pair\/ $(M_2,M_3)$ is one of the following$:$
\begin{align*}
&(V_{\mu+L},V_{\la+\mu+L}), \quad \mu\in L^*, \;\; 2\mu\not\in L,\\
&(V_{\mu+L}^{\epsilon_1},V_{\la+\mu+L}^{\epsilon_2}),\quad \mu\in L^*, \;\; 2\mu\in L, \;\;\epsilon_1\in\{\pm\}, \;\;
\epsilon_2=\epsilon_1\epsilon\pi_{\la,2\mu}, \\
&(V_L^{T_\chi, \,\epsilon_1}, V_L^{T^{(\la)}_{\chi}, \,\epsilon_2}),\quad \epsilon_1\in\{\pm\}, \;\; \epsilon_2=c_\chi(\la)\epsilon_1\epsilon.
\end{align*}
In all other cases, the fusion rules of type\/ $\displaystyle\binom{M_3}{V_{\la+L}^{\epsilon}\;M_2}$ are zero.
\end{theorem}

%%%%%%%%%%%%%%%%%%%%%%%%%%%%%
\section{Classification of irreducible modules% over the orbifold vertex algebra
}\label{sirr}
%%%%%%%%%%%%%%%%%%%%%%%%%%%%%

In this section, we prove our main result, the classification of all irreducible modules over the orbifold vertex algebra $V_Q^\si$.
As before, $Q$ is a positive-definite even integral lattice and $\si$ is an isometry of $Q$ of order $2$. 

%%%%%%%%%%%%%%%%%%%%%%%%%%%%%%%%%%%%%%%
\subsection{The sublattice $\bar Q$}\label{sqbar}
%%%%%%%%%%%%%%%%%%%%%%%%%%%%%%%%%%%%%%%
The map $\si$ is extended by linearity to the complex vector space $\lieh=\CC\otimes_\ZZ Q$.
We will denote by 
\begin{equation}\label{qbar1}
\pi_\pm=\frac{1}{2}(1\pm\si) \,, \qquad \al_\pm = \pi_\pm(\al)
\end{equation}
the projections onto the eigenspaces of $\si$.
Introduce the important sublattices
\begin{equation}\label{Lpm}
L_\pm=\lieh_\pm\cap Q,\, \qquad L=L_+\oplus L_-\subseteq Q\,,
\end{equation}
where $\lieh_\pm=\pi_\pm(\lieh)$.
Note that $\lieh=\lieh_+\oplus\lieh_-$ is an orthogonal direct sum.

\begin{lemma}\label{le3.1}
We have\/ $\al_\pm\in (L_\pm)^* \subseteq L^*$ for any\/ $\al\in Q$.
\end{lemma}
\begin{proof}
Indeed,
\[
(\al_+|\be) = (\al_++\al_-|\be) =  (\al|\be) \in\ZZ
\]
for all $\al\in Q$, $\be\in L_+\subseteq Q$.
\end{proof}

Observe that %$\si$ acts trivially on the quotient\/ $Q/L$
%since $\al-\si\al=2\al_-\in L_-\subseteq L$ implies $\al+L=\si(\al)+L=\si(\al+L)$. Furthermore, we have 
$2\al_\pm\in L_\pm$ and $|\al_\pm|^2=\frac{1}{4}|2\al_\pm|^2\in\frac{1}{2}\ZZ$ for all $\al\in Q$.

\begin{lemma}\label{le3.2}
For\/ $\al\in Q$, the following are equivalent$:$
\begin{enumerate}[$(i)$]
\item $\si^2(e^\al)=e^\al,$
\item $|\al_\pm|^2\in\ZZ,$
\item $(\al|\si\al)\in 2\ZZ.$
\end{enumerate}
\end{lemma}

\begin{proof}
Note that $4|\al_\pm|^2=|\al\pm\si\al|^2=2|\al|^2\pm2(\al|\si\al)$, so that
\begin{equation*}
|\al_\pm|^2=\frac{1}{2}(\al|\si\al)\mod\ZZ.
\end{equation*}
This shows the equivalence between $(ii)$ and $(iii)$.

Using \eqref{twlat4}, we find $\si^2(e^\al)=\eta(\al)\eta(\si\al)e^\al$. On the other hand, 
by \eqref{lat22}, \eqref{twlat3} and \eqref{twlat2}, we have
\begin{equation*}
\eta(\al)\eta(\si\al) = \ep(\al,\si\al)\ep(\si\al,\al) = (-1)^{(\al|\si\al)}.
\end{equation*}
This proves the equivalence between $(i)$ and $(iii)$.
\end{proof}

From now on, we let 
\begin{equation}\label{qbar3}
\bar{Q}=\{\al\in Q \,|\, (\al|\si\al)\in 2\ZZ\} \,.
\end{equation}
It is clear that $\bar Q$ is $\si$-invariant.
\begin{lemma}\label{le3.3}
The subset\/ $\bar{Q}$ is a sublattice of\/ $Q$ of index\/ $1$ or\/ $2$.
\end{lemma}

\begin{proof}
For $\al, \be\in Q$, we have
\begin{align*}
(\al-\be|\si\al-\si\be) %&= (\al|\si\al)+(\be|\si\be)-(\al|\si\be)-(\be|\si\al)\\
=(\al|\si\al)+(\be|\si\be) \mod 2\ZZ \,,
\end{align*}
since 
\begin{equation*}
(\al|\si\be) = (\si\al|\si^2\be) = (\be|\si\al) \,.
\end{equation*}
Now if $\al, \be\in\bar Q$ or $\al, \be\notin\bar Q$, then $\al-\be\in\bar{Q}$.
\end{proof}

As a consequence of Lemmas \ref{le3.1} and \ref{le3.2}, we obtain:
\begin{corollary}\label{co3.4}
The lattices\/ $\ZZ\al_\pm+L$ are integral for all\/ $\al\in\bar Q$.
\end{corollary}

By definition, we have $(V_Q)^{\si^2}=V_{\bar{Q}}$, and
\begin{equation}\label{qbar2}
V_Q^\si = \bigl((V_Q)^{\si^2} \bigr)^\si = V_{\bar Q}^\si \,.
\end{equation}
Therefore, we may assume that $|\si|=2$ on $V_Q$ and only work with the sublattice $\bar{Q}$. For simplicity, we use $Q$ instead of $\bar{Q}$ for the rest of this section.

%%%%%%%%%%%%%%%%%%%%%%%%%%%%%
\subsection{Restricting the orbifold $V_Q^\si$ to $V_L^\si$}\label{srestr}
%%%%%%%%%%%%%%%%%%%%%%%%%%%%%

By \cite{FHL, LL}, we have that the subalgebra $V_L$ of $V_Q$ is isomorphic to the
tensor product $V_{L_+}\otimes V_{L_-}$, since $L=L_+\oplus L_-$ is an orthogonal direct sum.
Note that $\si$ acts as the identity operator on $L_+$ and as $-1$ on $L_-$. Then
$V_L^\si\cong V_{L_+}\otimes V_{L_-}^+$ is a subalgebra of $V_Q^\si$. 

\begin{proposition}\label{decomp}
Every\/ $V_Q^\si$-module is a direct sum of irreducible\/ $V_L^\si$-modules. In particular, $V_Q^\si$ has this form.
\end{proposition}

\begin{proof}
It is shown in Theorem 3.16 of \cite{DLM3} that the vertex algebra $V_{L_+}$ is regular, since $L_+$ is positive definite.
It is also shown in \cite{A2,ABD, DJL} that the vertex algebra $V_{L_-}^+$ is regular.
Since the tensor product of regular vertex algebras is again regular (Proposition 3.3 in \cite{DLM3}), we have that 
$V_L^\si\cong V_{L_+}\otimes V_{L_-}^+$ is also regular. 
\end{proof}

In order to obtain a precise description of $V_Q^\si$, we will decompose $V_Q$ as 
a direct sum of irreducible modules over $V_L^\si$.
This is done in two steps. The first step is to break $V_Q$ as a direct sum of $V_L$-modules, using the cosets
of $Q$ modulo $L$.
Since the lattice $Q$ is integral, we have $Q\subseteq L^*$ and we can view $Q/L$ as a subgroup of $L^*/L$. It follows that
\[
V_Q=\displaystyle\bigoplus_{\ga+L\in Q/L} V_{\ga+L} \,,
\]
where each $V_{\ga+L}$ is an irreducible $V_L$-module \cite{D1}. 
Writing
$\ga=\ga_{+}+\ga_{-}$, we get
\[
V_{\ga+L}\cong V_{\ga_++L_+}\otimes V_{\ga_-+L_-}
\]
as modules over $V_L\cong V_{L_+}\otimes V_{L_-}$.
Therefore,
\begin{equation}\label{orbsplit}
V_Q\cong\bigoplus_{\ga+L\in Q/L} V_{\ga_++L_+}\otimes V_{\ga_-+L_-}
\end{equation}
as $V_L$-modules. Note that $\ga_++L_+$ and $\ga_-+L_-$ depend only on the coset $\ga+L$ and not on the representative $\ga\in Q$,
because $\al_\pm\in L_\pm$ for $\al\in L$.

Since $\si\ga_-=-\ga_-$ and $2\ga_-\in L_-$, it follows that
$\si$ acts on the $V_{L_-}$-module $V_{\ga_-+L_-}$.
The second step is to decompose each module $V_{\ga_-+L_-}$ into eigenspaces for $\si$, 
which are irreducible as $V_{L_-}^+$-modules \cite{AD}.
We thus obtain the following description of $V_Q^\si$.

\begin{proposition}\label{orb}
%Let $Q/L=\{\ga_0+L, \ga_1+L,\ldots,\ga_r+L\}$, where $\ga_0=0$. Then 
The orbifold\/ $V_Q^\si$ decomposes as a direct sum of irreducible\/ $V_L^\si$-modules as follows$:$
\begin{eqnarray}
V_Q^\si\cong\bigoplus_{\ga+L\in Q/L} V_{\ga_++L_+}\otimes V_{\ga_-+L_-}^{\eta(\ga)}\,.
\end{eqnarray}
%where $\eta$ is given by \eqref{twlat4}.
\end{proposition}

\begin{proof}
Using \eqref{orbsplit}, it is enough to determine the subspace $S_\ga$ of $\si$-invariants in $V_{\ga_++L_+}\otimes V_{\ga_-+L_-}$,
for a fixed $\ga\in Q$. As a $V_L^\si$-module, $S_\ga$ is generated by the element (cf.\ \eqref{twlat4}):
\begin{equation}\label{vgamma}
v_\ga=e^\ga+\eta(\ga)e^{\si\ga}=e^{\ga_+}\otimes (e^{\ga_-}+\eta(\ga)e^{-\ga_-}) \in V_{\ga_++L_+}\otimes V_{\ga_-+L_-}^{\eta(\ga)} \,.
\end{equation}
Since $V_{\ga_++L_+}\otimes V_{\ga_-+L_-}^{\eta(\ga)}$ is an irreducible $V_L^\si$-module, it must be equal to $S_\ga$.
\end{proof}

From the study of tensor products in \cite{FHL, LL},
all irreducible modules over $V_L^\si \cong V_{L_+} \otimes V_{L_-}^+$ are tensor products of irreducible modules over the factors $V_{L_+}$ and $V_{L_-}^+$. 
By the results of  \cite{D1,DN,AD} reviewed in Sections \ref{twlat} and \ref{tcasesi-1}, we obtain that all irreducible $V_L^\si$-modules have the form:
\begin{enumerate}
\item $V_{\la+L_+}\otimes V_{\mu+L_-}$, where $\la\in L_+^*$, $\mu\in L_-^*$, $2\mu\not\in L_-$,
\item $V_{\la+L_+}\otimes V_{\mu+L_-}^\pm$, where $\la\in L_+^*$, $\mu\in L_-^*$, $2\mu\in L_-$,
\item $V_{\la+L_+}\otimes V_{L_-}^{T,\pm}$, where $\la\in L_+^*$,
\end{enumerate}
and $T$ is an irreducible module for the group $G_\si$ associated to the lattice $L_-$.
We refer to the $V_L^\si$-modules obtained from untwisted $V_L$-modules as {\it orbifold modules of untwisted type} and the ones obtained from twisted $V_L$-modules as {\it orbifold modules of twisted type}.

%%%%%%%%%%%%%%%%%%%%%%%%%%%%%%%
\subsection{Irreducible Modules over $V_Q^\si$}
%%%%%%%%%%%%%%%%%%%%%%%%%%%%%%%

In this subsection, we present our main result, the explicit classification of irreducible $V_Q^\si$-modules. As a consequence, we will find all of them as submodules of twisted or untwisted $V_Q$-modules. Recall that, by \eqref{qbar2}, we can assume that $Q=\bar Q$.

\begin{theorem}\label{result1}
Let\/ $Q$ be a positive-definite even lattice, and\/ $\sigma$ be an automorphism of\/ $Q$ of order two
such that\/ $(\al|\si\al)$ is even for all\/ $\al\in Q$. Then as a module over\/ $V_L^\si\cong V_{L_+}\otimes V_{L_-}^+$
each irreducible\/ $V_Q^\sigma$-module is isomorphic to one of the following$:$
\begin{align}
\label{AD1}
&\bigoplus_{\ga+L\in Q/L}V_{\ga_++\la+L_+}\otimes V_{\ga_-+\mu+L_-}
\qquad (2\mu\not\in L_-)\,, \\
\label{AD2}
&\bigoplus_{\ga+L\in Q/L}V_{\ga_++\la+L_+}\otimes V_{\ga_-+\mu+L_-}^{\epsilon\eta(\ga)}
\qquad (2\mu\in L_-)\,, \\
\label{AD3}
&\bigoplus_{\ga+L\in Q/L}V_{\ga_++\la+L_+}\otimes V_{L_-}^{T_\chi^{(\ga_-)},\epsilon_\ga}
\,,
\end{align}
where\/ $\la\in L_+^*$, $\mu\in L_-^*$, $\epsilon\in\{\pm\}$, 
$\chi$ is a central character for the group\/ $G_\si$ associated to\/ $L_-$, and\/ 
$\epsilon_\ga=\epsilon\eta(\ga)c_\chi(\ga_-)$.
\end{theorem}

\begin{proof}
Let $W$ be an irreducible $V_Q^\sigma$-module. Then $W$ is a $V_L^\si$-module by restriction and, by Proposition \ref{decomp}, $W$ is a direct sum of irreducible $V_L^\si$-modules. 
Suppose $A\subseteq W$ is an irreducible $V_L^\si$-module, and define $A^{(\ga)}$ from $A$ as follows.
If 
\[
A=V_{\la+L_+}\otimes V_{\mu+L_-}, \quad V_{\la+L_+}\otimes V_{\mu+L_-}^\epsilon, \quad\text{or}\quad
V_{\la+L_+}\otimes V_{L_-}^{T_\chi,\epsilon}, 
\]
then 
\begin{align*}
A^{(\ga)} = V_{\la+\ga_++L_+}\otimes V_{\mu+\ga_-+L_-}, \quad
&V_{\la+\ga_++L_+} \otimes V_{\mu+\ga_-+L_-}^{\epsilon\eta(\ga)}, \\
&\text{and}\quad  V_{\la+\ga_++L_+}\otimes V_{L_-}^{T_\chi^{(\ga_-)},\epsilon_\ga}, 
\end{align*}
respectively,
where $\epsilon\in\{\pm\}$ and $\epsilon_\ga=\epsilon\eta(\ga)c_\chi(\ga_-)$.
We will consider separately the untwisted and twisted types.

Let $A$ be of untwisted type, i.e., one of the modules $V_{\la+L_+}\otimes V_{\mu+L_-}$ for $2\mu\not\in L_-$, 
or $V_{\la+L_+}\otimes V_{\mu+L_-}^\epsilon$ for $2\mu\in L_-$.
Let $B\subseteq W$ be another irreducible $V_L^\si$-module that is possibly of twisted type.
By Proposition \ref{orb}, $V_Q^\si$ is a direct sum of irreducible $V_L^\si$-modules generated by the vectors $v_\ga$ from 
\eqref{vgamma}, where $\ga\in Q$.
By restricting the field $Y(v_\ga,z)$ to $A$ and then projecting onto $B$, we obtain an intertwining operator of $V_L^\si$-modules of type $\displaystyle\binom{B}{V_{\ga+L}^{\eta(\ga)}\;\;A}$. From the study of intertwining operators in \cite{ADL}, we have that the intertwining operator $Y(v_\ga,z)$ can be written as the tensor product
\[
Y(v_\ga,z)=Y(e^{\ga_+},z)\otimes Y(e^{\ga_-}+\eta(\ga)e^{-\ga_-},z),
\]
where $Y(e^{\ga_+},z)$ is an intertwining operator of type $\displaystyle\binom{V_{\la'+L_+}}{V_{\ga_++L_+}\;\;V_{\la+L_+}}$ and $Y(e^{\ga_-}+\eta(\ga)e^{-\ga_-},z)$ is an intertwining operator of type\\ $\displaystyle\binom{V_{\mu'+L_-}}{V_{\ga_-+L_-}^{\eta(\ga)}\;\;V_{\mu+L_-}}$ or of type $\displaystyle\binom{V_{\mu'+L_-}^{\pm\eta(\ga)}}{V_{\ga_-+L_-}^{\eta(\ga)}\;\;V_{\mu+L_-}^\pm}$.
By \cite{DL1}, the fusion rules for $Y(e^{\ga_+},z)$ are zero unless $\la'=\la+\ga_+$.
Since $\ga_-\in L_-^*\cap\frac{1}{2}L_-$ and $|\ga_-|^2\in\ZZ$, we have that $\pi_{\ga_-,2\mu}=1$ when $2\mu\in L_-$
(cf.\ \eqref{pi}). Hence the fusion rules for $Y(e^{\ga_-}+\eta(\ga)e^{-\ga_-},z)$ are zero unless $\mu'=\mu+\ga_-$, by Theorem \ref{ADL}.
Therefore, for $\ga+L\in Q/L$, we have $B=A^{(\ga)}$. Hence $A\subseteq W$ implies that $\bigoplus_{\ga\in Q/L}A^{(\ga)}\subseteq W$. Since $W$ is irreducible, we obtain that
\[
W=\bigoplus_{\ga+L\in Q/L}A^{(\ga)}.
\]

Now let $A=V_{\la+L_+}\otimes V_{L_-}^{T_\chi,\pm}$ and $B\subseteq W$ be another irreducible $V_L^\si$-module that is possibly of untwisted type. As above, the field $Y(v_\ga,z)$ gives rise to an intertwining operator of $V_L^\si$-modules of type $\displaystyle\binom{B}{V_{\ga+L}^{\eta(\ga)}\;\;A}$ and can be written as the tensor product
\[
Y(v_\ga,z)=Y(e^{\ga_+},z)\otimes Y(e^{\ga_-}+\eta(\ga)e^{-\ga_-},z),
\]
where $Y(e^{\ga_+},z)$ is an intertwining operator of type $\displaystyle\binom{V_{\la'+L_+}}{V_{\ga_++L_+}\;\;V_{\la+L_+}}$ and $Y(e^{\ga_-}+\eta(\ga)e^{-\ga_-},z)$ is an intertwining operator of type\\ $\displaystyle\binom{V_{L_-}^{T_{\chi'},\pm\epsilon}}{V_{\ga_-+L_-}^{\eta(\ga)}\;\;V_{L_-}^{T_\chi,\pm}}$, where $\epsilon=c_\chi(\ga)\eta(\ga)$. As with the untwisted type, the fusion rules for $Y(e^{\ga_+},z)$ are zero unless $\la'=\la+\ga_+$.
 By Theorem \ref{ADL}, the action of $Y(e^{\ga_-},z)$ on $V_{L_-}^{T_\chi}$ is determined by computing $c_\chi(\ga_-)$ (cf. \eqref{c}) and is zero unless $\chi'=\chi^{(\ga_-)}$. Since the lattice $\ZZ\ga_-+L_-$ is integral (by \coref{co3.4}), the map $\ep$ can be extended to this lattice with values $\pm 1$. Therefore $\ep(\ga_-,2\ga_-)=\ep(\ga_-,\ga_-)^2=1$ and \eqref{c} becomes
$
c_\chi(\ga_-)=s(2\ga_-);
$
see \eqref{chiual}.
Hence the eigenspace of each summand in the $V_Q^\si$-module may change depending on the signs of each $U_{2\ga_-}$.
Therefore, $B=A^{(\ga)}$ and
\[
W=\bigoplus_{\ga+L\in Q/L}A^{(\ga)}.
\]
This completes the proof.
\end{proof}

\begin{theorem}\label{result2}
Let\/ $Q$ be a positive-definite even lattice, and\/ $\sigma$ be an automorphism of\/ $Q$ of order two
such that\/ $(\al|\si\al)$ is even for all\/ $\al\in Q$.  Then every irreducible\/ $V_Q^\sigma$-module is a submodule of a\/ $V_Q$-module or a\/ $\si$-twisted\/ $V_Q$-module.
\end{theorem}

\begin{proof}
Let us consider first the untwisted case.
By Theorem \ref{result1}, any irreducible $V_Q^\si$-module $W$ of untwisted type is given by \eqref{AD1} or \eqref{AD2}
when considered as a $V_L^\si$-module by restriction.
For a fixed $\ga\in Q$,
the nonzero fusion rules for $Y(e^{\ga_+},z)$ and $Y(e^{\ga_-}+\eta(\ga)e^{-\ga_-},z)$ are equal to $1$ and the intertwining operators in \cite{ADL} are given by the usual formula \eqref{lat5} up to a scalar multiple. %This determines the action of each $e^\ga$ on $W$.
Using that 
\[
(\ga|\la+\mu)-(\si\ga|\la+\mu)=(2\ga_-|\la+\mu)=(2\ga_-|\mu)\in\ZZ \,,
\]
we have that for some $m\in\ZZ$,
\begin{align*}
Y(v_\ga,z)  e^{\la+\mu} &=
Y(e^\ga+\eta(\ga)e^{\si\ga},z) e^{\la+\mu} \\
&= z^{(\ga|\la+\mu)}\bigl(E(\ga,z)e^{\ga+\la+\mu}+\eta(\ga)z^mE(\si\ga,z)e^{\si\ga+\la+\mu}\bigr) \,,
\end{align*}
where 
\[
E(\al,z)=\exp\Bigl( \sum_{n<0} \al_n \frac{z^{-n}}{-n} \Bigr)\exp\Bigl( \sum_{n>0} \al_n \frac{z^{-n}}{-n} \Bigr)
\]
contains only integral powers of $z$. This implies that $(\la+\mu|\ga)\in\ZZ$ for all $\ga \in Q$, i.e., $\la+\mu\in Q^*$.
Then for the untwisted $V_Q$-module $V_{\la+\mu+Q}$, we have that
\begin{align*}
V_{\la+\mu+Q}&=\bigoplus_{\ga+L\in Q/L}V_{\ga+\la+\mu+L}\\
&\cong\bigoplus_{\ga+L\in Q/L}V_{\ga_++\la+L_+}\otimes V_{\ga_-+\mu+L_-}\\
\end{align*}
as a direct sum of irreducible $V_L$-modules. 
Using the intertwining operators in \cite{ADL}, we see that $W$ is a submodule of the restriction of $V_{\la+\mu+Q}$ to $V_Q^\si$.

Now we consider the twisted case.
By Theorem \ref{result1}, any irreducible $V_Q^\si$-module $W$ of twisted type is given by \eqref{AD3} when considered as a $V_L^\si$-module by restriction.
Then, for $\ga\in Q$, the nonzero fusion rules for $Y(e^{\ga_+},z)$ and $Y(e^{\ga_-}+\eta(\ga)e^{-\ga_-},z)$ are equal to $1$ and the intertwining operators in \cite{ADL} are given by the usual formula \eqref{twlat12} up to a scalar multiple. Since these scalars can be absorbed in $U_\ga$, we will have \eqref{twlat12} without loss of generality. Therefore, the action of $Y(e^\ga,z)$ on $W$ can be determined,
so that its modes are linear maps 
\[
V_{\la+L_+}\otimes V_{L_-}^{T_\chi,\epsilon}\to V_{\ga_++\la+L_+}\otimes V_{L_-}^{T_\chi^{(\ga_-)},\epsilon_\ga}  \qquad (\epsilon\in\{\pm\}) \,,
\]
where $\epsilon_\ga=\epsilon c_\chi(\ga)\eta(\ga)$ (cf.\ \eqref{AD3}).
Hence the $\si$-twisted $V_Q$-module is given by 
\[
%T(\la,\chi)=
\bigoplus_{\ga+L\in Q/L}V_{\ga_++\la+L_+}\otimes V_{L_-}^{T_\chi^{(\ga_-)}}
\]
and its restriction to $V_Q^\si$ contains $W$. 
\end{proof}

%%%%%%%%%%%%%%%%%%%%%%%%%%%%%%%%%%%%%%
\section{Root lattices and Dynkin diagram automorphisms}\label{sroot}
%%%%%%%%%%%%%%%%%%%%%%%%%%%%%%%%%%%%%%

In this section, we present examples of the lattice $Q$ being a root lattice of type ADE, corresponding to the simply-laced simple Lie algebras. We use the classification from \cite{D1, DN, AD} and the construction of twisted modules from Section \ref{twlat} to construct explicitly all irreducible $V_Q^\si$-modules. In each case, a correspondence between the two constructions is shown.
In order to apply Theorems \ref{result1} and \ref{result2}, we first calculate $\bar{Q}$ and $L$. Then the twisted $V_{L}$-modules are found. When necessary, the intertwiners from \cite{ADL} are used to construct the $V_Q^\si$-modules. The untwisted and twisted $V_{\bar Q}$-modules are calculated using $(\bar{Q})^*/\bar{Q}$ and its $\si$-invariant elements. 
For more details and additional examples, the reader is referred to \cite{E}.

%%%%%%%%%%%%%%%%%%%%%%%%%%%%%%%%%%%%%%%
\subsection{$A_2$ root lattice with a Dynkin diagram automorphism}\label{A2D}

Consider the $A_2$ simple roots $\{\alpha_1, \alpha_2\}$,
the associated root lattice $Q=\mathbb{Z}\alpha_1+\mathbb{Z}\alpha_2$, and the Dynkin diagram automorphism $\sigma\colon \alpha_1 \leftrightarrow \alpha_2$. 
The case of $A_n$ root lattice for even $n$ is similar and is treated in \cite{E}.

Set $\alpha=\alpha_1+\alpha_2$ and $\beta=\alpha_1-\alpha_2$. Then
\[
|\alpha |^2=2 \,, \qquad |\beta |^2=6 \,, \qquad (\alpha |\beta)=0 \,,
\]
and
\[
L_+=\ZZ\al \,, \qquad L_-=\ZZ\be \,, \qquad \bar{Q}=L= L_+\oplus L_- \,.
\]
%Note that $L$ is written as an orthogonal decomposition and the quotient $\bar{Q}/L$ is trivial. 
Therefore, %by Proposition \ref{orb}, 
\begin{equation*}%\label{orb2}
V_Q^\si =V_{\bar Q}^\si=V_L^\si \cong V_{\ZZ\al}\otimes V_{\ZZ\be}^+ \,.
\end{equation*}
Using the results of \cite{FHL,D1,DN}, we obtain a total of $20$ distinct irreducible $V_L^\si$-modules:
\begin{align*}
&V_{\mathbb{Z}\alpha}\otimes V_{\mathbb{Z}\beta}^\pm \,, &\qquad& V_{\frac{\alpha}{2}+\mathbb{Z}\alpha}\otimes V_{\mathbb{Z}\beta}^\pm \,,\\
&V_{\mathbb{Z}\alpha}\otimes V_{\frac{\beta}{2}+\mathbb{Z}\beta}^\pm \,, &\qquad& V_{\frac{\alpha}{2}+\mathbb{Z}\alpha}\otimes V_{\frac{\beta}{2}+\mathbb{Z}\beta}^\pm \,,\\
&V_{\mathbb{Z}\alpha}\otimes V_{\frac{\beta}{6}+\mathbb{Z}\beta} \,, &\qquad& V_{\frac{\alpha}{2}+\mathbb{Z}\alpha}\otimes V_{\frac{\beta}{6}+\mathbb{Z}\beta} \,,\\
&V_{\mathbb{Z}\alpha}\otimes V_{\frac{\beta}{3}+\mathbb{Z}\beta} \,, &\qquad& V_{\frac{\alpha}{2}+\mathbb{Z}\alpha}\otimes V_{\frac{\beta}{3}+\mathbb{Z}\beta} \,,\\
&V_{\mathbb{Z}\alpha}\otimes V_{\mathbb{Z}\beta}^{T_j, \pm} \,, &\qquad& V_{\frac{\alpha}{2}+\mathbb{Z}\alpha}\otimes V_{\mathbb{Z}\beta}^{T_j, \pm}\qquad (j=1,2) \,,
\end{align*}
where $T_1,T_2$ denote the two irreducible modules over the group $G_\si$ associated to the lattice $\ZZ\be$. They are $1$-dimensional and on them $U_\beta=\pm\ii/64$ by \eqref{chiual}.

On the other hand, by Dong's Theorem \cite{D1}, the irreducible $V_L$-modules have the form $V_{\la+L}$ $(\la+L\in L^*/L)$. If $\si(\la+L)=\la+L$, the module $V_{\la+L}$ breaks into eigenspaces $V_{\la+L}^\pm$ of $\si$;
otherwise, $\si$ is an isomorphism of $V_L^\si$-modules $V_{\la+L} \to V_{\si(\la)+L}$.
Thus, there are 12 distinct irreducible $V_L^\si$-modules of untwisted type:
\begin{align*}
& V_L^\pm \,, \quad V_{\frac{\al}2+L}^\pm \,, \quad V_{\frac{\be}2+L}^\pm \,, \quad V_{\frac{\al}2+\frac{\be}2+L}^\pm \,, \\
& V_{\frac{\be}6+L} \,, \quad V_{\frac{\be}3+L} \,, \quad V_{\frac{\al}2+\frac{\be}6+L} \,, \quad V_{\frac{\al}2+\frac{\be}3+L} \,.
\end{align*}
The correspondence of these modules to those above is given by:
\[
V_{m\al+n\be+L} \cong V_{m\al+L_+} \otimes V_{n\be+L_-}  \qquad (m,n\in\QQ) \,.
\]

Now we will construct the $\si$-twisted $V_L$-modules using Section \ref{twlat}. Consider the irreducible modules over the group $G_\si$ associated to the lattice $L$.
One finds that on them $U_\beta=\pm\ii/64$ and $U_\al$ acts freely. Hence, such modules can be identified with the space $P=\mathbb{C}[q,q^{-1}]$, where
$U_{\alpha}$ acts as a multiplication by $q$. Then on $P$ we have:
\begin{equation*}
e^{\pi\ii\alpha_{(0)}} = s_1 \,, \qquad
U_{\alpha} = q \,, \qquad
U_\beta = s_2\frac{\ii}{64} \,,
\end{equation*}
where the signs $s_1,s_2\in\{\pm\}$ are independent. The corresponding four $G_\si$-modules will be denoted as $P_{(s_1,s_2)}$.
The irreducible $\si$-twisted $V_L$-modules have the form $\F_\si\otimes P_{(s_1,s_2)}$, where $\F_\si$ is the $\si$-twisted Fock space
(see \eqref{twheis2}).
Next, we restrict the $\si$-twisted $V_L$-modules to $V_L^\si$.
The automorphism $\si$ acts on each $P_{(s_1,s_2)}$ as the identity operator, since
\[
\si(q^n)=\sigma(U_{\alpha}^n\cdot 1)=U_{\alpha}^n\cdot 1=q^n \,.
\]
We obtain $8$ irreducible $V_L^\si$-modules of twisted type:
\[
\F_\si^\pm\otimes P_{(s_1,s_2)} \,, \qquad s_1,s_2\in\{\pm\} \,,
\]
where $\F_\si^\pm$ are the $\pm1$-eigenspaces of $\si$.
We have the following correspondence among irreducible $V_L^\si$-modules of twisted type:
\begin{align*}
\F_\si^\pm\otimes P_{(+,+)} &\cong V_{\mathbb{Z}\alpha}\otimes V_{\mathbb{Z}\beta}^{T_1, \pm}, &\qquad
\F_\si^\pm\otimes P_{(-,+)} &\cong V_{\frac{\al}{2}+\mathbb{Z}\alpha}\otimes V_{\mathbb{Z}\beta}^{T_1, \pm},\\
\F_\si^\pm\otimes P_{(+,-)} &\cong V_{\mathbb{Z}\alpha}\otimes V_{\mathbb{Z}\beta}^{T_2, \pm}, &\qquad
\F_\si^\pm\otimes P_{(-,-)} &\cong V_{\frac{\al}{2}+\mathbb{Z}\alpha}\otimes V_{\mathbb{Z}\beta}^{T_2, \pm}.
\end{align*}

%%%%%%%%%%%%%%%%%%%%%%%%%%%%%%%%%%%%%%%
\subsection{$A_2$ root lattice with the negative of a Dynkin diagram automorphism}\label{A2nD}

Consider now the negative Dynkin diagram automorphism $\phi=-\si$. 
Keeping the notation from the previous subsection, we have: $L_+=\ZZ\be$, $L_-=\ZZ\al$, and 
\begin{equation*}%\label{A2orb-}
V_Q^\phi\cong V_{\ZZ\be}\otimes V_{\ZZ\al}^+ \,.
\end{equation*}
Furthermore, $L_+^*/L_+= \ZZ\frac{\be}6/\ZZ\be$ and $\bigl(L_-^*/L_-\bigr)^\phi=L_-^*/L_- = \ZZ\frac{\al}2/\ZZ\al$. 
Thus, using the results of \cite{FHL,D1,DN}, we obtain a total of $48$ distinct irreducible $V_Q^\phi$-modules:
\begin{equation*}
V_{i\frac{\be}6+\mathbb{Z}\be}\otimes V_{\mathbb{Z}\al}^\pm\,, \qquad V_{i\frac{\be}6+\mathbb{Z}\be}\otimes V_{\frac{\al}{2}+\mathbb{Z}\al}^\pm\,, \qquad V_{i\frac{\be}6+\mathbb{Z}\be}\otimes V_{\mathbb{Z}\al}^{T_j, \pm}\,,
\end{equation*}
where $i=0,\dots,5$ and $j=1,2$.

%%%%%%%%%%%%%%%%%%%%%%%%%%%%%%%%%%%%%%%%
\subsection{$A_3$ root lattice with a Dynkin diagram automorphism}\label{A3D}

Consider the $A_3$ simple roots $\{\alpha_1, \alpha_2, \alpha_3\}$,
the root lattice $Q=\mathbb{Z}\alpha_1+\mathbb{Z}\alpha_2+\mathbb{Z}\alpha_3$, and the Dynkin diagram automorphism $\sigma\colon \alpha_1 \leftrightarrow \alpha_3$, $\alpha_2 \leftrightarrow \alpha_2$. 
The case of $A_n$ root lattice for odd $n$ is similar and is treated in \cite{E}.

Set $\alpha=\alpha_1+\alpha_3$ and $\beta=\alpha_1-\alpha_3$. 
%Then $\al$, $\al_2$ and $\be$ are eigenvectors for $\si$ with eigenvalues $1$, $1$ and $-1$, respectively. 
%Also, $(\alpha |\alpha)=4=(\beta |\beta)$ and $(\alpha |\beta)=0$. Since $(\alpha_1 | \alpha_3) = 0$ and $(\alpha_2| \alpha_2) = 2$, we %have that $\al_1,\al_2,\al_3\in\bar{Q}$. Therefore 
Then $\bar{Q}=Q$ and
\begin{equation*}
L_+=\mathbb{Z}\alpha+\mathbb{Z}\alpha_2 \,, \qquad
L_-=\mathbb{Z}\beta \,, \qquad
Q/L=\{L,\alpha_1+L\} \,. %\label{Q/LA3}
\end{equation*}
Hence, by Proposition \ref{orb}, 
\begin{equation*}
V_Q^\sigma\cong\bigl(V_{L_+}\otimes V_{\mathbb{Z}\beta}^+\bigr)\oplus\bigl(V_{\frac{\alpha}{2}+L_+}\otimes V_{\frac{\beta}{2}+\mathbb{Z}\beta}^+\bigr).
\end{equation*}
It follows from Theorems \ref{ADL} and \ref{result1} that the irreducible $V_Q^\si$-modules are given by:
\begin{align*}
&\bigl(V_{L_+}\otimes V_{\mathbb{Z}\beta}^\pm\bigr)\oplus\bigl(V_{\frac{\alpha}{2}+{L_+}}\otimes V_{\frac{\beta}{2}+\mathbb{Z}\beta}^\pm\bigr)\,, \\
&\bigl(V_{L_+}\otimes V_{\frac{\beta}{2}+\mathbb{Z}\beta}^\pm\bigr)\oplus\bigl(V_{\frac{\al}{2}+L_+}\otimes V_{\mathbb{Z}\beta}^\pm\bigr) \,,\\
&\bigl(V_{\frac{\alpha_2}{2}+{L_+}}\otimes V_{\frac{\beta}{4}+\mathbb{Z}\beta}\bigr)\oplus\bigl(V_{\frac{\alpha_2}{2}+\frac{\alpha}{2}+{L_+}}\otimes V_{\frac{\beta}{4}+\mathbb{Z}\beta}\bigr)\,\\
&\bigl(V_{L_+}\otimes V_{\mathbb{Z}\beta}^{T_1,\pm}\bigr)\oplus\bigl(V_{\frac{\alpha}{2}+{L_+}}\otimes V_{\mathbb{Z}\beta}^{T_1,\pm}\bigr) \,, \\
&\bigl(V_{L_+}\otimes V_{\mathbb{Z}\beta}^{T_2,\pm}\bigr)\oplus\bigl(V_{\frac{\alpha}{2}+{L_+}}\otimes V_{\mathbb{Z}\beta}^{T_2,\mp}\bigr) \,,
\end{align*}
where $U_\be=(-1)^j/16$ on $T_j$ for $j=1,2$ (see \eqref{chiual}).

On the other hand, by Dong's Theorem \cite{D1}, the irreducible $V_Q$-modules have the form $V_{\la+Q}$ $(\la+Q\in Q^*/Q)$.
We have
\[
Q^*/Q= \{Q, \la_1+Q, 2\la_1+Q, 3\la_1+Q\} \,,
\]
where
\[
\lambda_1=\frac{1}{4}(3\alpha_1+2\alpha_2+\alpha_3) \,.
\]
Note that $Q$ and $2\la_1+Q = \frac{\al}2+Q$ are $\si$-invariant, while $\si(\la_1+Q)=3\la_1+Q$.
Thus, there are $5$ distinct irreducible $V_Q^\si$-modules of untwisted type:
\[
V_Q^\pm \,,\qquad V_{\frac{\al}2+Q}^\pm \,,\qquad V_{\lambda_1+Q}  \,.
\]
We have the following correspondence:
\begin{align*}
V_Q^\pm &\cong \bigl(V_{L_+}\otimes V_{\mathbb{Z}\beta}^\pm\bigr)\oplus\bigl(V_{\frac{\alpha}{2}+{L_+}}\otimes V_{\frac{\beta}{2}+\mathbb{Z}\beta}^\pm\bigr)\,, \\
V_{\frac{\al}2+Q}^\pm &\cong \bigl(V_{L_+}\otimes V_{\frac{\beta}{2}+\mathbb{Z}\beta}^\pm\bigr)\oplus\bigl(V_{\frac{\al}{2}+L_+}\otimes V_{\mathbb{Z}\beta}^\pm\bigr) \,,\\
V_{\lambda_1+Q} &\cong \bigl(V_{\frac{\alpha_2}{2}+{L_+}}\otimes V_{\frac{\beta}{4}+\mathbb{Z}\beta}\bigr)\oplus\bigl(V_{\frac{\alpha_2}{2}+\frac{\alpha}{2}+{L_+}}\otimes V_{\frac{\beta}{4}+\mathbb{Z}\beta}\bigr)\,.
\end{align*}

We now construct the $\si$-twisted $V_Q$-modules using Section \ref{twlat}. Consider the irreducible modules over the group $G_\si$ associated to the lattice $Q$.
We find from \eqref{chiual} that on them $U_\beta=\pm1/16$, and there are two such modules, $P_1$ and $P_2$, corresponding to $U_\beta=-1/16$ and $U_\beta=1/16$,
respectively. For both $j=1,2$, we can identify $P_j=\mathbb{C}[q, q^{-1},p, p^{-1}]$ so that
\[
U_{\alpha_1}=q \,, \qquad
U_{\alpha_2}=p(-1)^{q\frac{\partial}{\partial q}}\,.
\]
Then on $P_j$ we have:
\[
U_{\alpha_3}=(-1)^{j+1}  q \,, \qquad
U_\alpha=(-1)^{j+1} 4q^2\,, \qquad
U_\beta=\frac{(-1)^j}{16} \,.
\]
The automorphism $\si$ acts on each of these modules: $\si$ is the identity operator on $P_1$, while on $P_2$ we have $\si=(-1)^{q\frac{\partial}{\partial q}}$.
Hence, $P_2$ decomposes into two eigenspaces $P_2^\pm$ with eigenvalues $\pm1$.
The irreducible $\si$-twisted $V_Q$-modules have the form $\F_\si\otimes P_j$ for $j=1,2$, where $\F_\si$ is the $\si$-twisted Fock space
(see \eqref{twheis2}). Since $\F_\si$ itself decomposes into $\pm 1$-eigenspaces of $\si$, 
we obtain $4$ distinct irreducible $V_Q^\si$-modules of twisted type:
\begin{align*}
\bigl(\F_\si\otimes P_1\bigr)^\pm &= \F_\si^\pm\otimes P_1 \,,\\
\bigl(\F_\si\otimes P_2\bigr)^\pm &= \bigl(\F_\si^\pm\otimes P_2^+\bigr)\oplus\bigl(\F_\si^\mp\otimes P_2^-\bigr) \,.
\end{align*}
We have the following correspondence:
\begin{align*}
\bigl(\F_\si\otimes P_1\bigr)^\pm &\cong \bigl(V_{L_+}\otimes V_{\mathbb{Z}\beta}^{T_1,\pm}\bigr)\oplus\bigl(V_{\frac{\alpha}{2}+{L_+}}\otimes V_{\mathbb{Z}\beta}^{T_1,\pm}\bigr) \,, \\
\bigl(\F_\si\otimes P_2\bigr)^\pm &\cong \bigl(V_{L_+}\otimes V_{\mathbb{Z}\beta}^{T_2,\pm}\bigr)\oplus\bigl(V_{\frac{\alpha}{2}+{L_+}}\otimes V_{\mathbb{Z}\beta}^{T_2,\mp}\bigr) \,.
\end{align*}

%%%%%%%%%%%%%%%%%%%%%%%%%%%%%%%%%%%%%%%
\subsection{$A_3$ root lattice with the negative of a Dynkin diagram automorphism}\label{A3nD}

Consider now the negative Dynkin diagram automorphism, $\phi=-\si$. With the notation from the previous subsection, we have: 
\[
L_+=\ZZ\be \,, \qquad L_-=\ZZ\al+\ZZ\al_2\,, \qquad \bigl(L_-^*/L_-\bigr)^\phi=L_-^*/L_- \,,
\]
and
\begin{equation*}%\label{A3orb-}
V_Q^\phi\cong \bigl(V_{\mathbb{Z}\beta}\otimes V_{L_-}^+\bigr)\oplus\bigl(V_{\frac{\be}{2}+\ZZ\be}\otimes V_{\frac{\al}{2}+L_-}^+\bigr)\,.
\end{equation*}
By Theorem \ref{result1}, 
the irreducible $V_Q^\phi$-modules of untwisted type are given by:
\begin{align*}
&\bigl(V_{\mathbb{Z}\beta}\otimes V_{L_-}^\pm\bigr)\oplus\bigl(V_{\frac{\be}{2}+\ZZ\be}\otimes V_{\frac{\al}{2}+L_-}^\pm\bigr) \,,\\
&\bigl(V_{\mathbb{Z}\beta}\otimes V_{\frac{\al}{2}+L_-}^\pm\bigr)\oplus\bigl(V_{\frac{\be}{2}+\ZZ\be}\otimes V_{L_-}^\pm\bigr) \,,\\
&\bigl(V_{\frac{\be}{4}+\mathbb{Z}\beta}\otimes V_{\frac{\al_2}{2}+L_-}^\pm\bigr)\oplus\bigl(V_{\frac{3\be}{4}+\ZZ\be}\otimes V_{\frac{\al+\al_2}{2}+L_-}^\pm\bigr) \,,\\
&\bigl(V_{\frac{\be}{4}+\mathbb{Z}\beta}\otimes V_{\frac{\al+\al_2}{2}+L_-}^\pm\bigr)\oplus\bigl(V_{\frac{3\be}{4}+\ZZ\be}\otimes V_{\frac{\al_2}{2}+L_-}^\pm\bigr) \,,
\end{align*}
and the ones of twisted type are:
\begin{align*}
&\bigl(V_{\ZZ\be}\otimes V_{L_-}^{T_1,\pm}\bigr)\oplus\bigl(V_{\frac{\be}{2}+\ZZ\be}\otimes V_{L_-}^{T_1,\pm}\bigr) \,, \\
&\bigl(V_{\ZZ\be}\otimes V_{L_-}^{T_2,\pm}\bigr)\oplus\bigl(V_{\frac{\be}{2}+\ZZ\be}\otimes V_{L_-}^{T_2,\mp}\bigr) \,.
\end{align*}

%%%%%%%%%%%%%%%%%%%%%%%%%%%%%%%%%%
\subsection{$D_n$ root lattice with a Dynkin diagram automorphism}\label{Dn}

Consider the $D_n$ simple roots $\{\alpha_1, \ldots, \alpha_n\}$, where $n\geq4$,
%The nondegenerate symmetric $\ZZ$-bilinear form $(\cdot |\cdot )$ is given by 
%\[ 
%((\al_i|\al_j))_{i,j}=\bigl( \begin{array}{cccccc}
%2 & -1 &&&& 0\\
%-1 & 2 & -1 &&& \vdots\\
%& \ddots & \ddots & \ddots && 0\\
%&& -1 & 2 & -1 & -1\\
%&&& -1 & 2 & 0\\
%0 & \cdots & 0 & -1 & 0 & 2\\
%\end{array} \bigr).
%\]
%\newline
the root lattice $Q$, %=\bigoplus_{i=1}^n\mathbb{Z}\alpha_i$, 
and the Dynkin diagram automorphism $\sigma\colon \alpha_{n-1} \leftrightarrow \alpha_n$ and $\alpha_i \leftrightarrow \alpha_i$ for $i=1,\ldots, n-2$. 
This case is similar to the $A_3$ case discussed in Section \ref{A3D}, so we will be brief (see \cite{E} for details).

Set $\alpha=\alpha_{n-1}+\alpha_n$ and $\beta=\alpha_{n-1}-\alpha_n$. 
Then 
\[
L_+=\mathbb{Z}\alpha+\sum_{i=1}^{n-2}\mathbb{Z}\alpha_i \,, \qquad L_-=\mathbb{Z}\beta \,, \qquad Q=\bar{Q} \,,
\]
and
\[
Q/L=\{L,\alpha_{n-1}+L\} \,.
\]
Hence, by Proposition \ref{orb}, 
\begin{equation*}
V_Q^\sigma\cong\bigl(V_{L_+}\otimes V_{\mathbb{Z}\beta}^+\bigr)\oplus\bigl(V_{\frac{\alpha}{2}+L_+}\otimes V_{\frac{\beta}{2}+\mathbb{Z}\beta}^+\bigr) \,.
\end{equation*}
By Theorems \ref{ADL} and \ref{result1},
the irreducible $V_Q^\si$-modules are given by:
\begin{align*}
&\bigl(V_{L_+}\otimes V_{\mathbb{Z}\beta}^\pm\bigr)\oplus\bigl(V_{\frac{\alpha}{2}+L_+}\otimes V_{\frac{\beta}{2}+\mathbb{Z}\beta}^\pm\bigr) \,,\\
&\bigl(V_{L_+}\otimes V_{\frac{\beta}{2}+\mathbb{Z}\beta}^\pm\bigr)\oplus\bigl(V_{\frac{\alpha}{2}+L_+}\otimes V_{\mathbb{Z}\beta}^\pm\bigr) \,,\\
&\bigl(V_{\frac{n-1}{4}\alpha+\theta+L_+}\otimes V_{\frac{\beta}{4}+\mathbb{Z}\beta}\bigr)\oplus\bigl(V_{\frac{n+1}{4}\alpha+\theta+L_+}\otimes V_{\frac{\beta}{4}+\mathbb{Z}\beta}\bigr) \,,\\
&\bigl(V_{L_+}\otimes V_{\mathbb{Z}\beta}^{T_1,\pm}\bigr)\oplus\bigl(V_{\frac{\alpha}{2}+{L_+}}\otimes V_{\mathbb{Z}\beta}^{T_1,\pm}\bigr) \,,\\
&\bigl(V_{L_+}\otimes V_{\mathbb{Z}\beta}^{T_2,\pm}\bigr)\oplus\bigl(V_{\frac{\alpha}{2}+{L_+}}\otimes V_{\mathbb{Z}\beta}^{T_2,\mp}\bigr) \,,
\end{align*}
where $\theta = \frac{1}{2}\sum_{i=0}^{\lfloor\frac{n-3}{2}\rfloor}\alpha_{2i+1}$ and $U_\beta=(-1)^j/16$ on $T_j$ for $j=1,2$.

On the other hand, note that
\[
Q^*/Q= \Bigl\{Q, \frac{\alpha}{2}+Q, \lambda_{n-1}+Q, \lambda_n+Q\Bigr\} \,,
\]
where
\[
\lambda_{n-1}=\frac{1}{2}\Bigl(\alpha_1+2\alpha_2+\cdots+(n-2)\alpha_{n-2}+\frac{1}{2}n\alpha_{n-1}+\frac{1}{2}(n-2)\alpha_n\Bigr)
\]
and $\la_n=\si(\la_{n-1})$.
Using Dong's Theorem \cite{D1}, we obtain $5$ distinct irreducible $V_Q^\si$-modules of untwisted type:
\[
V_Q^\pm\,,\qquad V_{\frac{\alpha}{2}+Q}^\pm\,,\qquad V_{\lambda_{n-1}+Q}\,.
\]
We have the following correspondence:
\begin{align*}
V_Q^\pm &\cong \bigl(V_{L_+}\otimes V_{\mathbb{Z}\beta}^\pm\bigr)\oplus\bigl(V_{\frac{\alpha}{2}+L_+}\otimes V_{\frac{\beta}{2}+\mathbb{Z}\beta}^\pm\bigr) \,,\\
V_{\frac{\alpha}{2}+Q}^\pm &\cong \bigl(V_{L_+}\otimes V_{\frac{\beta}{2}+\mathbb{Z}\beta}^\pm\bigr)\oplus\bigl(V_{\frac{\alpha}{2}+L_+}\otimes V_{\mathbb{Z}\beta}^\pm\bigr) \,,\\
V_{\lambda_{n-1}+Q} &\cong \bigl(V_{\frac{n-1}{4}\alpha+\theta+L_+}\otimes V_{\frac{\beta}{4}+\mathbb{Z}\beta}\bigr)\oplus\bigl(V_{\frac{n+1}{4}\alpha+\theta+L_+}\otimes V_{\frac{\beta}{4}+\mathbb{Z}\beta}\bigr) \,.
\end{align*}

We now construct the irreducible $\si$-twisted $V_Q$-modules using Section \ref{twlat}.
There are two irreducible $G_\si$-modules, $P_1$ and $P_2$, and both can be identified with $\mathbb{C}[p_1^{\pm1},\ldots, p_{n-1}^{\pm1}]$ as vector spaces. The action of $G_\si$
on $P_j$ is determined by:
\begin{align*}
U_{\alpha_i}&=p_i(-1)^{p_{i+1}\frac{\partial}{\partial p_{i+1}}} \,, \qquad i=1,\dots,n-2 \,,\\
U_{\alpha_{n-1}}&=p_{n-1} \,,\qquad U_{\alpha_n}=(-1)^{j+1} p_{n-1} \,,\\ 
U_\alpha&=(-1)^{j+1} 4p_{n-1}^2 \,, \qquad  U_\beta=\frac{(-1)^j}{16} \,.
\end{align*}
The automorphism $\si$ acts as the identity operator on $P_1$, and as $(-1)^{q\frac\d{\d q}}$ on $P_2$. 
We obtain $4$ distinct irreducible $V_Q^\si$-modules of twisted type:
\begin{align*}
\F_\si^\pm\otimes P_1 \,,\qquad
\bigl(\F_\si\otimes P_2\bigr)^\pm =
\bigl(\F_\si^\pm\otimes P_2^+\bigr)\oplus\bigl(\F_\si^\mp\otimes P_2^-\bigr) \,,
\end{align*}
and we have the following correspondence:
\begin{align*}
\bigl(\F_\si\otimes P_1\bigr)^\pm &\cong \bigl(V_{L_+}\otimes V_{\mathbb{Z}\beta}^{T_1,\pm}\bigr)\oplus\bigl(V_{\frac{\alpha}{2}+{L_+}}\otimes V_{\mathbb{Z}\beta}^{T_1,\pm}\bigr) \,, \\
\bigl(\F_\si\otimes P_2\bigr)^\pm &\cong \bigl(V_{L_+}\otimes V_{\mathbb{Z}\beta}^{T_2,\pm}\bigr)\oplus\bigl(V_{\frac{\alpha}{2}+{L_+}}\otimes V_{\mathbb{Z}\beta}^{T_2,\mp}\bigr) \,.
\end{align*}

%%%%%%%%%%%%%%%%%%%%%%%%%%%%%%%%%%%%%
\subsection{$E_6$ root lattice with a Dynkin diagram automorphism}\label{E6}

Consider the $E_6$ Dynkin diagram with the simple roots $\{\alpha_1,\ldots, \alpha_6\}$ labeled as follows:

\medskip

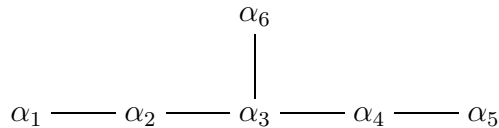
\centerline{\xymatrix{ 
& & \al_6 \ar@{-}[d] \\ 
\al_1 \ar@{-}[r] & \al_2 \ar@{-}[r] & \al_3 \ar@{-}[r] & \al_4 \ar@{-}[r] & \al_5 \\
}} 

\bigskip

Let $Q$ be the root lattice, %$Q=\bigoplus_{i=1}^6\mathbb{Z}\alpha_i$. 
and $\si$ be the Dynkin diagram automorphism $\alpha_1 \leftrightarrow \alpha_5$, $\alpha_2 \leftrightarrow \alpha_4$, with fixed points $\alpha_3$ and $\alpha_6$. 
%Then $\si$ is also an automorphism of $Q$. 
Set 
\[
\alpha^1=\alpha_1+\alpha_5 \,,\quad
\beta^1=\alpha_1-\alpha_5 \,,\quad
\alpha^2=\alpha_2+\alpha_4 \,,\quad
\beta^2=\alpha_2-\alpha_4 \,.
\]
Then we have: $Q=\bar{Q},$
\[
L_+=\ZZ\al^1+\ZZ\al^2+\mathbb{Z}\alpha_3+\mathbb{Z}\alpha_6 \,,\qquad
L_-=\mathbb{Z}\beta^1+\ZZ\be^2 \,,
\]
and
\[
Q/L=\{L,\al_1+L,\al_2+L, \al_1+\al_2+L\} \,.
\]
Hence, by Proposition \ref{orb}, 
\begin{align*}%\label{orbE6}
V_Q^\si\cong&\displaystyle\bigl(V_{L_+}\otimes V_{L_-}^+\bigr)
\oplus\bigl(V_{\frac{\al^1}{2}+L_+}\otimes V_{\frac{\be^1}{2}+L_-}^+\bigr)\nonumber\\
&\oplus\bigl(V_{\frac{\al^2}{2}+L_+}\otimes V_{\frac{\be^2}{2}+L_-}^+\bigr)
\oplus\bigl(V_{\frac{\al^1+\al^2}{2}+L_+}\otimes V_{\frac{\be^1+\be^2}{2}+L_-}^+\bigr) \,.
\end{align*}

There are $3$ distinct irreducible $V_Q^\si$-modules of untwisted type:
\begin{align*}
\bigl(V_{L_+} &\otimes V_{L_-}^\pm\bigr) \oplus\bigl(V_{\frac{\al^1}{2}+L_+}\otimes V_{\frac{\be^1}{2}+L_-}^\pm\bigr)\\
&\oplus\bigl(V_{\frac{\al^2}{2}+L_+}\otimes V_{\frac{\be^2}{2}+L_-}^\pm\bigr)\oplus\bigl(V_{\frac{\al^1+\al^2}{2}+L_+}\otimes V_{\frac{\be^1+\be^2}{2}+L_-}^\pm\bigr)\,,\\
\bigl(V_{L_+} &\otimes V_{\mu_2+L_-}\bigr)\oplus\bigl(V_{\frac{\al^1}{2}+L_+}\otimes V_{\mu_2+\frac{\be^1}{2}+L_-}\bigr)
\oplus\bigl(V_{\frac{\al^2}{2}+L_+}\otimes V_{\mu_2+\frac{\be^2}{2}+L_-}\bigr)\\
&\oplus\bigl(V_{\frac{\al^1+\al^2}{2}+L_+}\otimes V_{\mu_2+\frac{\be^1+\be^2}{2}+L_-}\bigr)\,,
\end{align*}
where $\mu_2=\frac{1}{3}(\be^1+2\be^2)$.
To describe the ones of twisted type, notice that the group $G_\si$ associated to $L_-$ is abelian and its characters $\chi$ are determined by $\chi(U_{\be^i})$ where $i=1, 2$.
The latter are given by \eqref{chiual} with $s(\be^i)=s_i\in\{\pm1\}$. Thus we obtain 
$8$ distinct irreducible $V_Q^\si$-modules of twisted type:
\begin{align*}
\bigl(V_{L_+}\otimes V_{L_-}^{T_\chi,\pm}\bigr)
&\oplus\bigl(V_{\frac{\al^1}{2}+L_+}\otimes V_{L_-}^{T_\chi,\pm s_1}\bigr)\\
&\oplus\bigl(V_{\frac{\al^2}{2}+L_+}\otimes V_{L_-}^{T_\chi,\pm s_2}\bigr)
\oplus\bigl(V_{\frac{\al^1+\al^2}{2}+L_+}\otimes V_{L_-}^{T_\chi,\pm s_1s_2}\bigr) \,.
\end{align*}

We now describe the irreducible $\si$-twisted $V_Q$-modules using Section \ref{twlat}. The irreducible modules over the
group $G_\si$ associated to $Q$ can be identified with the vector space
$P=\mathbb{C}[q_1^{\pm1},q_2^{\pm1},q_3^{\pm1}, q_4^{\pm1}]$, so that:
\begin{align*}
U_{\alpha_1}&=q_1(-1)^{\frac{\partial}{\partial q_2}} \,,\qquad
&U_{\alpha_2}&=q_2(-1)^{\frac{\partial}{\partial q_3}} \,,\\
U_{\alpha_3}&=q_3(-1)^{\frac{\partial}{\partial q_4}} \,,\qquad
&U_{\alpha_6}&=q_4 \,.
\end{align*}
In fact, there are $4$ such modules depending on the signs of $U_{\beta^i}$; for each $s_1,s_2\in\{\pm1\}$, we have:
\begin{align*}
U_{\alpha_4}&=-s_2 q_2(-1)^{\frac{\partial}{\partial q_3}} \,,\qquad
&U_{\alpha_5}&=-s_1 q_1(-1)^{\frac{\partial}{\partial q_2}} \,,\\
U_{\alpha^i}&=-s_i 4q_i^2 \,,\qquad
&U_{\beta^i}&=s_i\frac{1}{16} \,.
\end{align*}
The corresponding $G_\si$-module will be denoted $P_{\chi}$ where $\chi=(s_1,s_2)$.
The automorphism $\si$ acts on $P_{\chi}$ by:
\[
\si(q_i) = -s_i q_i \,, \qquad i=1,2 \,.
\]
The irreducible $\si$-twisted $V_Q$-modules have the form $\F_\si\otimes P_\chi$, and they give rise to $8$
irreducible $V_Q^\si$-modules of twisted type:
\[ 
\bigl(\F_\si\otimes P_\chi\bigr)^\pm \,, \qquad \chi=(s_1,s_2) \,, \;\; s_1,s_2\in\{\pm1\} \,.
\]
They correspond to the above modules of twisted type with the same characters $\chi$.

%%%%%%%%%%%%%%%%%%%%%%%%%%%%%%%%%%%
\section*{Acknowledgements}
The first author is grateful to Victor Kac and Ivan Todorov for valuable discussions and collaboration on
orbifolds of lattice vertex algebras.
\\

\textit{Note added in proof.} After our paper was submitted, we
learned of the preprint ``Cyclic permutations of lattice vertex
operator algebras" by C.~Dong, F.~Xu, and N.~Yu (arXiv:1501.00063),
where $2$-cycle permutation orbifolds of lattice vertex algebras are investigated. The
classification of irreducible modules of such orbifolds can be derived
as a special case of our results. We thank Nina Yu for bringing their
preprint to our attention.

\bibliographystyle{amsalpha}

\end{document}